\documentclass[12pt]{amsart}
\usepackage[all,cmtip]{xy}
\usepackage[all]{xy}
\usepackage{subcaption}
\usepackage{multirow}
\usepackage{ar}
\usepackage{amssymb}
\usepackage{algorithmicx}
\usepackage{leftindex}
\usepackage[english]{babel}
\usepackage{graphics}
\usepackage{graphicx}
\graphicspath{ {images/} }
\usepackage{epstopdf}

\usepackage{amsmath}
\usepackage{mathrsfs}
\usepackage{tikz-cd}
\usetikzlibrary{shapes,shadows,arrows}
\usepackage{cite}
\usepackage{placeins}
\usepackage{hyperref}
\calclayout                   
\newtheorem{dummy}{anything}[section]
\newtheorem{theorem}[dummy]{Theorem}

\newtheorem{lemma}[dummy]{Lemma}

\newtheorem{corollary}[dummy]{Corollary}

\theoremstyle{definition}
\newtheorem{definition}[dummy]{Definition}

\newtheorem{example}[dummy]{Example}

\newtheorem{remark}[dummy]{Remark}

\newtheorem{question}[dummy]{Question}

\newcommand{\bZ}{\mathbb Z}

\newcommand{\bR}{\mathbb R}

\newcommand{\ha}{\hookrightarrow}
 
\def\:{\mkern 1.2mu \colon}
\newcommand{\mmatrix}[4]{\left (\vcenter
	{\xymatrix@C-2pc@R-2pc{#1&#2\\#3&#4} } \right )}

\DeclareMathOperator{\im}{Im}

\numberwithin{equation}{section}
\begin{document}
	\title[ Elementary Methods for Persistent Homotopy Groups ]
	{Elementary Methods for Persistent Homotopy Groups}
	
	
	\author{ Henry Adams, Mehmet Al\.{I} Batan, Mehmetc{\.I}k Pamuk, Han\.{I}fe Varl{\i}}
	
	\address{Department of Mathematics
		\newline\indent
		University of Florida
		\newline\indent
		Gainesville, Florida 32611, USA}  
	\email{henry.adams{@}ufl.edu}
	
	\address{Department of Mathematics
		\newline\indent
		Middle East Technical University
		\newline\indent
		Ankara 06531, Turkey}  
	\email{alibatan{@}metu.edu.tr,  mpamuk{@}metu.edu.tr}	
	
	\address{Department of Mathematics 
		\newline\indent
		\c{C}ank{\i}r{\i} Karatekin University
		\newline\indent
		\c{C}ank{\i}r{\i} 18100, Turkey} 
	\email{hanifevarli@karatekin.edu.tr}

	\date{\today}	
\begin{abstract}

We study the foundational properties of persistent homotopy groups and develop elementary computational methods for their analysis.
Our main theorems are persistent analogues of the Van Kampen, excision, suspension, and Hurewicz theorems.
We prove a persistent excision theorem, derive from it a persistent Freudenthal suspension theorem, and obtain a persistent Hurewicz theorem relating the first nonzero persistent homotopy group of a space to its persistent homology.
As an application, we compute sublevelset persistent homotopy groups of alkane energy landscapes and show these invariants capture nontrivial loops and higher‑dimensional features that complement the information given by persistent homology.

\end{abstract}

	
	\maketitle


\section{Introduction}

	Topological data analysis is a recently emerging and fast-growing field for analyzing complex data using geometry and topology. Persistent homology is a powerful tool in topological data analysis for investigating data structure.  Persistent homology studies topological features of a space that persist for some range of parameter values.
	
	Let $X$ be a topological space.  By a filtration of $X$, $F_{X}$, we mean a family $\{X_k\}$ of increasing subsets of $X$ with respect to inclusions, i.e., 
	$X_k \subset X_l$ for $k < l \in \bR$.  Persistent homology is based on analyzing the homological changes occurring along the filtration.  It captures the topology of a filtration in terms of a multiset of intervals, called barcodes, corresponding to lifespans of topological features. This is done by considering the homomorphisms 
	$H_*(X_{kl}) \colon H_*(X_k)\to H_*(X_l)$ induced by the inclusion maps $X_{kl} \colon X_k\hookrightarrow X_l$.	
	Persistent homology has been developed as an algebraic method to study topological features of filtered spaces built from data, such as components, graph structures, holes, and voids.
    This theory has many applications and has become a central tool in topological data analysis \cite{carlsson, ghrist}.
	
	Although many notions from algebraic topology have been introduced into the setting of persistent homology, some of the most basic notions from homotopy theory 
	remain largely absent from the literature, with a few notable exceptions. In their pioneering work, Frosini and Mulazzani \cite{Frossini-Mulazzani} introduced and studied 
	size homotopy groups as an algebraic tool that allows one to obtain more efficient lower bounds for natural size distances.
    In~\cite{letscher}, instead of using homology, Letscher applied the homotopy functor to filtered spaces and defined persistent homotopy groups (see Section~\ref{2} for further details).
	Concurrent with our paper, M\'{e}moli and Zhou have studied persistent homotopy groups of metric spaces and their stability properties \cite{memoli-zhou}.
 
	Although homotopy groups are more challenging to work with than homology groups, they can also capture more information.
    For example, for knots, one can 
	work with the fundamental group of the complement \cite{gordon-luecke}, whereas the homology groups of the complement give no further information.  
	Letscher~\cite{letscher} applied persistent homotopy to detect if a complex is knotted and if that knotting can be unknotted in a larger complex.
    He also applied these techniques to analyze protein, RNA, and DNA structures.  Moreover, in~\cite{brendel}, persistent fundamental group calculations are carried out for knots arising from experimental 
	data on protein backbones.   
	
	One possible application of persistent homotopy groups is in image analysis.  Methods for image analysis have become an essential tool for many sciences, and topological data analysis has proven highly successful in aiding in a variety of such studies (for applications to environmental science, see for example~\cite{muszynski2019topological,smith,ver-hoef}).
    The theorems in our paper allow one to patch together topological information about smaller sub-images to obtain results about an entire image or to understand how various modifications (e.g.\ taking quotients) to a filtered complex change the persistence diagram (see~\cite[Section~5] {bleile-garin-heiss-maggs-robins}).

    In this paper, we apply persistent homotopy groups to the study of chemical energy landscapes.
    The conformation space of a molecule parametrizes its possible different shapes.
    The energy function on this conformation space determines which configurations are more likely and how the molecule might transition from one configuration to another.
    As one varies the energy threshold, the number of connected components in a sublevelset of the conformation space determines the number of essentially distinct conformations, and the number of 1-dimensional holes is related to the number of minimal energy paths between local energy minima.
    In \cite{Mirth2021}, the persistent homology diagrams of the alkane molecules were characterized, and in this paper, we explain how the persistent homotopy groups of these molecules provide even more information.

    Our paper aims to develop elementary methods for determining persistent homotopy groups.
    In Section~\ref{2}, we recall the definition of persistent homotopy groups, particularly persistent fundamental groups.
	Calculating persistent homotopy groups can be challenging, even at a single filtration level.
    In such cases, it might be better to consider a space $X$ as a union of its subspaces whose homotopy groups are easier to calculate. For this reason, we prove in Section~\ref{2} that the Van Kampen theorem remains valid for persistent fundamental groups.
    In Subsection~\ref{groupoid}, we define fundamental groupoids and discuss a persistent version of the Van Kampen theorem for fundamental groupoids.
    When $X=A\cup B$ is a filtered topological space equipped with a filtration $F_X$, these theorems show how to understand the persistent fundamental group or groupoid of the filtration on $X$ in terms of the induced filtrations  on $A$, $B$, and $A\cap B$. These filtrations are denoted by $F_A, F_B$ and $F_{A\cap B}$, respectively, which are the families $\{A_k\}$, $\{B_k\}$ and $\{(A\cap B)_k\}$ where $A_k:=X_k\cap A$, $B_k:=X_k\cap B$, and $(A\cap B)_k:=X_k\cap (A\cap B)$, for any $k\in \bR$.
    
    In Section~\ref{3}, following \cite[Section~3]{bubenik-scott}, we define the interleaving distance between persistent homotopy groups.  
	This section discusses the relationship between the Van Kampen theorem and the interleaving distance between persistent fundamental groups.
Let  $X$ and $X'$ be based topological spaces with filtrations $F_X$ and $G_{X'}$.
We show that if $X$ and $X'$ decompose as the union of path-connected open subsets $A$ and $B$, and $A'$ and $B'$ respectively, then the interleaving distance between persistent fundamental group functors $\pi_1F_{X}$ and $\pi_1G_{X'}$ (see Section \ref{2.2} for the definition) is bounded above by $ \textrm{max}~ \{( d_I(\pi_1F_{A}, \pi_1G_{A'}), d_I(\pi_1F_{B}, \pi_1G_{B'}), d_I(\pi_1F_{A\cap B}, \pi_1G_{A'\cap B'}) \}$, the maximum interleaving distance between the corresponding pieces (Corollary~\ref{cor:interleaving}).	
 
In Section~\ref{4}, we give an excision theorem for persistent homotopy groups (Theorem~\ref{Exc}).
From it, we derive from a persistent Freudenthal suspension theorem (Theorem~\ref{suspension}).
In \cite[Theorem 1.5]{memoli-zhou}, the authors prove a persistent version of the Hurewicz theorem for persistent fundamental groups.
In this section, we also prove a persistent version of the Hurewicz theorem for higher persistent homotopy groups (Theorem \ref{hurewicz}).

In Section~\ref{5}, we analyze the sublevelset persistent homotopy groups of the energy landscape of alkane molecules. We also explain the additional information these persistent homotopy groups contain beyond what persistent homology provides.

Our theorems apply most naturally when a filtration is given on a topological space $X=A\cup B$; that filtration then induces a filtration on $A$ and $B$.
For example, a real-valued function $f\colon X\to \bR$ induces a sublevelset filtration $\{f^{-1}(-\infty,k)\}_{k\in \bR}$ on $X$, and also (by restriction) on $A$ and $B$.
This setting is common in applications, for example when the function $f$ is a measurement such as temperature or pressure defined on a subset $X$ of Euclidean space, or for example when $f$ is an energy function defined on a configuration space $X$ of a chemical system, and then one decomposes $X$ into parts.
Another common application of topology is when $Z$ is a point cloud dataset, and one uses a Vietoris--Rips simplicial complex to measure the shape of the data~\cite{carlsson}.
We caution the reader since the Vietoris--Rips complex of a union could be larger than the union of Vietoris--Rips complexes: If dataset $Z$ is decomposed as $Z=Z_A\cup Z_B$ then we could have $X\coloneqq \mathrm{VR}(Z;r)\supsetneq \mathrm{VR}(Z_A;r)\cup \mathrm{VR}(Z_B;r)$, and hence in order to obtain $X=A\cup B$ one might have to choose $A$ and $B$ to be larger than $\mathrm{VR}(Z_A;r)$ and $\mathrm{VR}(Z_B;r)$, respectively.

          \vskip .3 cm	
		
	\noindent{\bf {Acknowledgements.}} The authors thank Claudia Landi, Facundo M\'{e}moli, and Ling Zhou for their comments on the second version of this paper.  
	This research was supported by the Scientific and Technological Research Council of Turkey (TUBITAK) [grants number 117F015 and 220N359 ].

\section{The Van Kampen Theorem for Persistent Fundamental Groups}
\label{2}
	
    \subsection{Persistent Homotopy Groups}
	
	Let $\mathbf{Top}_{\bullet}$ be the category of based topological spaces and basepoint-preserving continuous maps.  Let $\mathbf{R}$ denote the category whose objects are the real numbers $\bR$ and which admits a unique morphism $k\to l$ whenever $k\leq l$.  Let us also denote the functors from $\mathbf{R}$ to the category $\mathbf{Top}_{\bullet}$ by $\mathbf{Top_{\bullet}^{R}}$.  
	
    Throughout the paper, we regard filtered topological spaces (topological spaces admitting a filtration) as members of  $\mathbf{Top_{\bullet}^{R}}$; that is, we consider filtered topological spaces indexed by $\bR$.  For example, for a topological space $X$ and a real-valued function $f \colon X\to \bR$, the sublevelsets $f^{-1}((-\infty, k))= \{x\in X \ | \  f(x)<k \}$ can be assembled into a filtration $F_{X}\in \mathbf{Top_{\bullet}^{R}}$.
    For $k\in \bR$, we define $X_k \coloneqq F_X(k)\coloneqq f^{-1}((-\infty, k))$ and consider it as a topological space with the subspace topology.  For $k\leq l$, we define $X_{kl}$ to be the inclusion $X_k \hookrightarrow  X_l$, which is a continuous map.  We refer to such a filtration as the sublevelset filtration of $f$.
	
	Let \textbf{Gp} denote the category of groups and group homomorphisms between them.  Now, consider the homotopy group functor $\pi_{n}\colon\mathbf{Top_{\bullet}} \to \mathbf{Gp}$.  
	It induces a push-forward from $\mathbf{Top_{\bullet}^{R}}$ to $\mathbf{Gp^{R}}$ via post-composition.  
	\begin{definition}
		The category 
        of persistent groups is $\mathbf{Gp^{R}}$.
        Its objects are the functors $\mathbf{R}\to \mathbf{Gp}$, and the morphisms in 
        $\mathbf{Gp^{R}}(F,G)$ are the natural transformations $F\Rightarrow G$.
		The morphisms from $F$ to $G$ in 
        $\mathbf{Gp^{R}}$ admit a pointwise description as a collection of group homomorphisms 
		$\{\phi(k): F(k)\to G(k) \ | \ k\in \bR \}$ such that for any $k\leq l$ in $\bR$, we have the following commutative diagram (Figure \ref{cd}).
		\begin{figure}[hbt]
			\centering
			\begin{tikzcd}                                    
				F(k) \arrow[r, "F(k\to l)"] \arrow[d, "\phi(k)"]    
				& F(l) \arrow[d, "\phi(l)"] \\
				G(k) \arrow[r, "G(k\to l)"]
				& G(l)
			\end{tikzcd}
   \caption{Commutative diagram of group homomorphisms.}
		\label{cd}
		\end{figure}
		
	\end{definition}
	
	\subsection{The Van Kampen Theorem for Persistent Fundamental Groups}\label{2.2}
	
	Let $X$ be a topological space with a fixed basepoint $x_0$.
    Let $F_{X}\in \mathbf{Top_{\bullet}^{R}} $ be a filtration for $X$ with 
	$x_0 \in X_k$ for all $k\in \bR$.  
	
	We define the \textbf{$(k, l)$-persistent fundamental group} of $X$ with respect to the filtration $F_X$ to be the image of the group homomorphism $\pi_1 (X_{kl})\colon \pi_1 (X_{k}, x_0) \to \pi_1 (X_{l}, x_0) $, 
     induced by the inclusion $X_{kl}$.   Throughout the paper, for notational ease, we denote this image group by $\im \pi_1 (X_{kl})$. 
    Note that $\im \pi_1 (X_{kl})$ can be thought of as the fundamental group elements in $\pi_1 (X_{k}, x_0)$ that are still alive (i.e.,\ that persist) in $\pi_1 (X_{l}, x_0) $.
	
	Now, let us recall the Van Kampen theorem, which gives a method for computing the fundamental groups of spaces that can be decomposed into simpler subspaces (see \cite{hatcher} for a general version of the theorem and further details).      
	
	\begin{theorem} \label{vankampen}
		If $X=A \cup B $ with $A$, $B$, and $A\cap B$ open and path-connected, then the induced homomorphism 
		$$
		\Phi \colon  (\pi_1(A) \ast  \pi_1(B))/N \rightarrow \pi_1(X)
		$$ 
		is an isomorphism where $N$ is the normal subgroup generated by all elements of the form $\pi_1(A\cap B\ha A) (w)\cdot  \pi_1(A\cap B\ha B)(w)^{-1}$ for $w\in \pi_1(A\cap B)$; see the commutative diagram in Figure \ref{vk}.
	\end{theorem}

	\begin{figure}[hbt]
		\centering
		\begin{tikzcd}
			& \pi_1(A) \ar[dr, "\pi_1(A\ha X)"] 
			&
			&[1.5em] \\
			\pi_1(A \cap B) \ar[ur, "\pi_1(A\cap B\ha A)"] \ar[dr, "\pi_1(A\cap B\ha B)"']
			&
			& \pi_1(X) 
			&  \\
			& \pi_1(B) \ar[ur, "\pi_1(B\ha X)" ']
			&	
		\end{tikzcd}
		\caption{This diagram is commutative: $\displaystyle \pi_1(A\ha X) \pi_1(A\cap B\ha A) =  \pi_1(B\ha X) \pi_1(A\cap B\ha B)$.}
		\label{vk}
	\end{figure}

	In the remainder of this subsection, we show that the Van Kampen theorem is valid also for persistent fundamental groups. 
	
	\begin{remark}\label{basepoint}
	In this version of the Van Kampen theorem, we assume that the induced filtration for $A\cap B$ contains the chosen basepoint at each filtration level. 
	But even if one uses a filtration on a path-connected space $X$ constructed from finite samples from $X$, early filtration terms could be non-path-connected. The fundamental group could, at best, give information only on one path component of each term in the filtration. Using fundamental groupoids, it may be possible to capture information on many path components in each term of the filtration.  For this reason, after we give a proof of the Van Kampen theorem for persistent fundamental groups,  we then investigate fundamental groupoids.  
	\end{remark}

\begin{theorem}\label{persistentVK}
Let $X$ be a filtered topological space $F_X = \{X_k\}_{k \in \mathbb{R}}$ such that $X = A \cup B$, where $A$, $B$, and $A \cap B$ are 
open subsets containing the chosen basepoint $x_0$ at each filtration level. Assume that the filtration $F_X$ induces filtrations $F_A$, $F_B$, and $F_{A \cap B}$ on $A$, $B$, and $A \cap B$, respectively, defined by $A_k = X_k \cap A$, $B_k = X_k \cap B$, and $(A \cap B)_k = X_k \cap (A \cap B)$.  Furthermore, assume that for all $k\in \mathbb{R}$, the subspaces $A_k$, $B_k$, and $(A\cap B)_k$ are path-connected. 
Fix $k \leq l \in \mathbb{R}$. 
Let 
\[
\Phi_{kl} : \im \pi_1(A_{kl}) * \im \pi_1(B_{kl}) \longrightarrow \im \pi_1(X_{kl})
\]
be the group homomorphism induced by the inclusion maps. Then the kernel of $\Phi_{k,l}$ is the normal subgroup $N_{kl}$ generated by all elements of the form
\[ i^A_{kl}(w) \cdot i^B_{kl}(w)^{-1} \]
for $w \in \im \pi_1((A \cap B)_{kl})$.
Consequently, $\Phi_{kl}$ induces an isomorphism
\[
\im \pi_1(X_{kl}) \cong \left( \im \pi_1(A_{kl}) * \im \pi_1(B_{kl}) \right) \big/ N_{kl}.
\]
\end{theorem}

\begin{proof}
Let $Y \subseteq Z \subseteq X$ be subspaces, and for each $k \leq l \in \mathbb{R}$, let $\pi_1(Y_{kl})$ denote the homomorphism $\pi_1(Y_k) \to \pi_1(Y_l)$ induced by inclusion, and let $\pi_1(Y_k \hookrightarrow Z_k)$ denote the inclusion-induced map between fundamental groups at filtration level $k$.

From the classical Van Kampen theorem, for each $k \in \mathbb{R}$, the map
\[
\Phi_k : \pi_1(A_k) * \pi_1(B_k) \to \pi_1(X_k),
\]
induced by the inclusion maps $\pi_1(A_k \hookrightarrow X_k)$ and $\pi_1(B_k \hookrightarrow X_k)$, is surjective and yields an isomorphism
\[
\pi_1(X_k) \cong \left( \pi_1(A_k) * \pi_1(B_k) \right) / N_k,
\]
where $N_k$ is the normal subgroup generated by all elements of the form
\[
\pi_1((A \cap B)_k \hookrightarrow A_k)(w) \cdot \pi_1((A \cap B)_k \hookrightarrow B_k)(w)^{-1}, \quad \text{for } w \in \pi_1((A \cap B)_k).
\]

Let $k \leq l$.
Since we have the Van Kampen theorem at each filtration level for fundamental groups, we have the commutative diagram given in Figure~\ref{fig:firstdiagram}. 

\begin{figure}[htb]
                 \xymatrix@!C@C30pt @!R@R-20pt{
                 & & \pi_1 (A_k)\ar@/^/[dr]^{{\small \pi_1(A_k\ha X_k)}}  \ar@{-->}@<2ex>@/^1.7pc/[dddd]^(.2){{\small \pi_1(A_{kl})}}& \\
                 &  \pi_1 ((A\cap B)_k) \ar@/^/[ur]^(.45){{\small \pi_1((A\cap B)_k\ha A_k)}}  \ar@/_/[dr]_(.45){{\small \pi_1((A\cap B)_k\ha B_k)}}  \ar@{-->}[dddd]_{\pi_1((A\cap B)_{kl})} &   &  \pi_1 (X_k) \ar@{-->}[dddd]^{\pi_1(X_{kl})}\\
                 & & \pi_1 (B_k) \ar@/_/[ur]_{{\small \pi_1(B_k\ha X_k)}} \ar@{-->}@<-2ex>@/_1.65pc/[dddd]_(0.25){{\small \pi_1(B_{kl})}}& \\
                 &&  &\\
                 & & \pi_1 (A_l) \ar@/^/[dr]^{{\small \pi_1(A_l\ha X_l)}} & \\
                 &  \pi_1 ((A\cap B)_l)  \ar@/^/[ur]^(.45){{\small \pi_1((A\cap B)_l\ha A_l)}}  \ar@/_/[dr]_(.4){{\small \pi_1((A\cap B)_l\ha B_l)}}  &   &  \pi_1 (X_l) \\
                 & & \pi_1 (B_l)  \ar@/_/[ur]_{{\small \pi_1(B_l\ha X_l)}} &  }
                 \caption{Fundamental group homomorphisms between levels $k$ and $l$.}
                 \label{fig:firstdiagram}
                  \end{figure}

Now we consider the restrictions of the homomorphisms $\pi_1((A \cap B)_l) \to \pi_1(A_l)$ and $\pi_1((A \cap B)_l) \to \pi_1(B_l)$ to the image of $\pi_1((A \cap B)_{kl})$.
Similarly, we consider the restriction of the homomorphism $\pi_1(A_l)\to \pi_1(X_l)$ to the image of $\pi_1(A_{kl})$, and the restriction of the homomorphism $\pi_1(B_l)\to \pi_1(X_l)$ to the image of $\pi_1(B_{kl})$.
We denote these restriction homomorphisms by $i^A_{kl}$, $i^B_{kl}$, $j^A_{kl}$, $j^B_{kl}$, as drawn in the commutative diagram given in Figure~\ref{pvk}. 

\begin{figure}[hbt]
		\centering
		\begin{tikzcd}
			& \im \pi_1(A_{kl}) \ar[dr, "j^A_{kl}"] 
			&
			&[1.5em] \\
			\im \pi_1((A \cap B)_{kl}) \ar[ur, "i^A_{kl}"] \ar[dr, "i^B_{kl}"']
			&
			& \im \pi_1 (X_{kl}) 
			&  \\
			& \im \pi_1(B_{kl}) \ar[ur, "j^B_{kl}" ']
			&	
		\end{tikzcd}
		\caption{This diagram is commutative: $\displaystyle j^A_{kl}\circ i^A_{kl} =  j^B_{kl}\circ i^B_{kl}$.}
		\label{pvk}
	\end{figure}

These induce a homomorphism
\[
\Phi_{kl} : \im \pi_1(A_{kl}) * \im \pi_1(B_{kl}) \longrightarrow \im \pi_1(X_{kl}).
\]

We now prove that $\Phi_{kl}$ is surjective and that its kernel is the normal subgroup $N_{kl}$ generated by all elements of the form
\[i^A_{kl}(w) \cdot i^B_{kl}(w)^{-1}\] 
for $ w \in \im \pi_1((A \cap B)_{kl})$.

\textbf{Surjectivity:} Let $x \in \im \pi_1(X_{kl})$. Then there exists $x_k \in \pi_1(X_k)$ such that $\pi_1(X_{kl})(x_k) = x$. Since $\Phi_k$ is surjective, there exist elements $a_k^{(i)} \in \pi_1(A_k)$ and $b_k^{(i)} \in \pi_1(B_k)$ such that:
\[
x_k = \Phi_k(\Pi_{i=1}^n(a_k^{(i)} b_k^{(i)})).
\]
Naturality of the inclusion maps (see Figure~\ref{fig:firstdiagram}) gives:
\begin{align*}
\pi_1(X_{kl})(\Phi_k(\Pi_{i=1}^n(a_k^{(i)} b_k^{(i)}))) &= \Phi_l(\Pi_{i=1}^n (\pi_1(A_{kl})(a_k^{(i)}) \cdot \pi_1(B_{kl})(b_k^{(i)}))) \\
&= \Phi_{kl}(\Pi_{i=1}^n(a_l^{(i)} b_l^{(i)})),
\end{align*}
where $a_l^{(i)} = \pi_1(A_{kl})(a_k^{(i)})$ and $b_l^{(i)} = \pi_1(B_{kl})(b_k^{(i)})$.
Thus, $x$ lies in the image of $\Phi_{kl}$.

\textbf{Kernel:} Define $N_{kl}$ to be the normal subgroup of $\im \pi_1(A_{kl}) * \im \pi_1(B_{kl})$ generated by
\[
i^A_{kl}(w) \cdot i^B_{kl}(w)^{-1}\]
for $ w \in \im \pi_1((A \cap B)_{kl})$.
By Figure \ref{pvk}, we have
\[\displaystyle j^A_{kl}\circ i^A_{kl} =  j^B_{kl}\circ i^B_{kl}.\]
So these relations are precisely those that are killed in $\Phi_{kl}$.
Thus, $\ker \Phi_{kl} = N_{kl}$.

\end{proof}

	\subsection{The Van Kampen Theorem for Persistent Fundamental Groupoids} \label{groupoid}
	Let $X$ be a topological space and $p$ and $q$ be a pair of points in $X$.
         The fundamental groupoid $\Pi_1(X)$ of $X$ is a category whose objects are the points of $X$, and whose morphisms from $p$ to $q$ are the homotopy classes of paths in $X$
         from $p$ to $q$ (relative to endpoints).
         These morphisms are denoted Mor($p, q$).
         The classical fundamental group $\pi_1(X,x_0)$ at the basepoint $x_0$, as discussed in the previous subsection, is Mor($x_0, x_0$) in $\Pi_1(X)$.
         
         \begin{theorem}  [Van Kampen theorem for the fundamental groupoid]
         Let $X$ be a topological space and $A$ and $B$ be two open subsets of $X$ such that $X=A\cup B$. Then, the following diagram, in which all morphisms are induced by inclusions of spaces, is a pushout square of groupoids:
         \begin{figure}[h]
			\centering
			\begin{tikzcd}      
				\Pi_1(A\cap B) \arrow[r] \arrow[d]    
				& \Pi_1(A) \arrow[d] \\
				\Pi_1(B) \arrow[r]
				& \Pi_1(X)
			\end{tikzcd}
            \caption{Pushout square of fundamental groupoids.}
		\end{figure}
         \end{theorem}
         
         We refer the reader to \cite{brown} for the proof and more details on fundamental groupoids.
         
         We define the \textbf{$(k, l)$-persistent fundamental groupoid} of $X$ with respect to the filtration $F_{X}$ to be the image of the groupoid homomorphism $\Pi(X_{kl}):\Pi(X_k)\to \Pi(X_l)$ induced by the inclusion of $X_k$ into $X_l$.  Throughout this subsection, we denote this persistent fundamental groupoid by $\im \Pi(X_{kl})$.
         We emphasize that we do not 
	choose a basepoint in this setup.
	
Since we have the Van Kampen theorem at each filtration level for fundamental groupoids, we have the commutative diagram in Figure \ref{fig:groupoid}, where the top and bottom squares are pushouts.
	
	   \begin{figure}[htb]
                 \xymatrix@!C@C30pt @!R@R-20pt{
                 & & \Pi_1 (A_k)\ar@/^/[dr]^{{\small \Pi_1(A_k\ha X_k)}}  \ar@{-->}@<2ex>@/^1.7pc/[dddd]^(.2){{\small \Pi_1(A_{kl})}}& \\
                 &  \Pi_1 ((A\cap B)_k) \ar@/^/[ur]^(.45){{\small \Pi_1({(A\cap B)}_k\ha A_k)}}  \ar@/_/[dr]_(.45){{\small \Pi_1({(A\cap B)}_k\ha B_k)}}  \ar@{-->}[dddd]_{\Pi_1((A\cap B)_{kl})} &   &  \Pi_1 (X_k) \ar@{-->}[dddd]^{\Pi_1(X_{kl})}\\
                 & & \Pi_1 (B_k) \ar@/_/[ur]_{{\small \Pi_1(B_k\ha X_k)}} \ar@{-->}@<-2ex>@/_1.65pc/[dddd]_(0.25){{\small \Pi_1(B_{kl})}}& \\
                 && &\\
                 & & \Pi_1 (A_l) \ar@/^/[dr]^{{\small \Pi_1(A_l\ha X_l)}} & \\
                 &  \Pi_1 ((A\cap B)_l)  \ar@/^/[ur]^(.45){{\small \Pi_1({(A\cap B)}_l\ha A_l)}}  \ar@/_/[dr]_(.4){{\small \Pi_1({(A\cap B)}_l\ha B_l)}}  &   &  \Pi_1 (X_l) \\
                 & & \Pi_1 (B_l)  \ar@/_/[ur]_{{\small \Pi_1(B_l\ha X_l)}} &  }
                 \caption{The fundamental groupoid morphisms between levels $k$ and $l$.}
                 \label{fig:groupoid}
                  \end{figure}

    As remarked earlier, it is certainly possible that a filtration of a space need not be path-connected at each stage in the filtration.
    Using fundamental groupoids, it may be possible to capture information on many path components in each term of the filtration.

Section~\ref{2.2} establishes that the classical Van Kampen theorem, a cornerstone of algebraic topology, extends naturally to the persistent setting.
By decomposing a filtered topological space \( X = A \cup B \) into open, path-connected subspaces \( A \), \( B \), and \( A \cap B \), the persistent fundamental group of \( X \) can be computed as a quotient of the free product of the persistent fundamental groups of \( A \) and \( B \), modulo relations arising from their intersection.
This result, formalized in Theorem~\ref{persistentVK}, enables localized computations of persistent homotopy groups by breaking down complex spaces into simpler components.
Fundamental groupoids may furthermore be useful in capturing information across multiple path components.

The key takeaway is that persistent homotopy inherits the decompositional power of the Van Kampen theorem, offering a systematic framework for analyzing filtered spaces, an advancement for applications in topological data analysis where such decompositions are often necessary.

\section{Interleaving Distance and the Van Kampen Theorem}
\label{3}

In this section, following  \cite{chazal} and \cite{bubenik-scott}, we define the interleaving distance between persistent homotopy groups.
Then, we look at the relationship between the Van Kampen theorem 
	and the interleaving distance.
	
	Let $F, G \in 
    \mathbf{Gp^{R}}$
    be two functors.  We say that $F$ and $G$ are $\delta$-\textit{interleaved} for some $\delta \geq 0$ if there exist two families of morphisms 
	$\{\phi(k): F(k)\to G(k+\delta) \ | \ k\in \bR\}$ and $\{\varphi(k): G(k)\to F(k+\delta) \ | \ k\in \bR\}$ such that the diagrams in Figure \ref{G_u} commute for all $k\leq l$ (see \cite{bubenik-scott} 
	for a more general definition and details). 
	
	\begin{figure}[hbt]
		\centering
		\begin{tikzcd}                                    
			F(k) \arrow[rr, "F(k\to l)"] \arrow[d, "\phi(k)"]    
			&& F(l) \arrow[d, "\phi(l)"] \\
			G(k + \delta) \arrow[rr, "G(k+\delta\to l+\delta)" ]
			&& G(l + \delta)
		\end{tikzcd}
		\begin{tikzcd}
			G(k) \arrow[rr, "G(k\to l)"] \arrow[ "\varphi(k)", d]    
			&& G(l) \arrow[d, "\varphi(l)"] \\
			F(k + \delta) \arrow[rr, "F(k+\delta\to l+\delta)"]
			&& F(l + \delta)
		\end{tikzcd}
		
		\vspace*{5mm}
		
		\begin{tikzcd}[column sep=.5em]
			& G(k+ \delta) \arrow[dr,"\varphi(k+\delta)"] \\
			F(k) \arrow[ur,"\phi(k)"] \arrow[rr, "F(k\to k+2\delta)"] && F(k+ 2\delta)
		\end{tikzcd}	
		\begin{tikzcd}[column sep=.5em]
			& F(k+ \delta) \arrow[dr,"\phi(k+\delta)"] \\
			G(k) \arrow[ur,"\varphi(k)"] \arrow[rr, "G(k\to k+2\delta)"] && G(k+ 2\delta)
		\end{tikzcd}
		\caption{Commutative diagrams for the $\delta$-interleaved functors $F$ and $G$.}
		\label{G_u}
	\end{figure}
	
	This induces the following extended pseudometric (see \cite[Theorem 3.3]{bubenik-scott}), the \textit{interleaving distance} between 
	$F$ and $G$, which is defined as 	
	$$
	d_I(F, G)=\textrm{inf} \{\delta \geq 0 \ | \ F \ \textrm{and} \ G \ \textrm{are} \ \delta\textrm{-interleaved} \}.
	$$
	\noindent
	We set $d_I(F, G)=\infty$ if $F$ and $G$ are not $\delta$-interleaved for any  $\delta \geq 0$.
	
	We say that the functors $F$ and $G$ are \emph{isomorphic} if there is a family of  isomorphisms $\{\phi(k): F(k)\to G(k) \ | \ k\in \bR\}$ 
	in the commutative diagram in Figure \ref{fig:iso} for all $k\le l$.
	\begin{figure}[h]\centering
		\begin{tikzcd}                                    
			F(k) \arrow[rr, "F(k\to l)"] \arrow[ "\phi(k)", d]    
			&& F(l) \arrow[d, "\phi(l)"] \\
			G(k ) \arrow[rr, "G(k\to l)"]
			&& G(l)
		\end{tikzcd}
        \caption{Commutative diagram for isomorphic functors $F$ and $G$.}
        \label{fig:iso}
	\end{figure}
		
	The interleaving distance $d_I$ defined above is indeed an extended pseudometric since it can take the value $\infty$, and since $d_I(F, G) = 0$ does not imply that 
	$F$ and $G$ are isomorphic.  But, if we identify functors whose interleaving distance is $0$, then $d_I$ is an extended
	metric on this set of equivalence classes (see \cite[Section~3]{bubenik-scott}).

	Let $X, X'\in \mathbf{Top_{\bullet}}$ with basepoints $x_0, x'_0$, respectively.  Let  $F_{X}, G_{X'}\in \mathbf{Top_{\bullet}^{R}}$ be  filtrations for these spaces.
Assume that  $X=A\cup B$ and $X'=A'\cup B'$  are covered by open and path-connected subsets such that $x_0 \in A\cap B$ and $x'_0 \in A'\cap B'$. 
    Since we have assumed that $x_0 \in X_k$ and $x'_0 \in X'_k$ for all $k\in \bR$, this implies that $x_0$ is in $A_k$, $B_k$, and $(A\cap B)_k$, and $x'_0$ is in $A'_k$, $B'_k$, and $(A'\cap B')_k$ for all $k \in \bR$.

\begin{remark}
In Theorem~\ref{theorem:iso}, Theorem~\ref{thm:interleaving}, and Corollary~\ref{cor:interleaving}, we assume that for each filtration level $k \in \mathbb{R}$, the subspaces $A_k, B_k$, $(A \cap B)_k$ and $A_k', B_k'$, $(A' \cap B')_k$ are all path-connected.
This ensures the well-definedness of the interleaving maps and the normal subgroups $N_k$ and $N_k'$ in the Van Kampen quotients.
These assumptions are needed for the persistence functors $\pi_1 F_A$, $\pi_1 F_B$, etc., to be meaningful at each level.
\end{remark}

 
    
	
For any $k\in \bR$, we have the following isomorphism coming from the Van Kampen theorem at level $k$ of our filtrations: 
	$$
	\displaystyle  \pi_1(X_k)\cong  (\pi_1(A_k)\ast \pi_1(B_k)) /N_k. 
	$$
	\noindent
	Here, $N_k$ is the normal subgroup generated by all the elements of the form 
	$$
	\pi_1({(A\cap B)}_k \ha A_k)(w)\cdot\pi_1({(A\cap B)}_k \ha B_k)(w)^{-1} 
	$$
	\noindent 
	for $w\in \pi_1((A\cap B)_k)$.  Let us define a functor $\displaystyle (\pi_1F_A*\pi_1F_B)/N \in \mathbf{Gp^{R}}$
    by 
	$$
	((\pi_1F_A*\pi_1F_B)/N )(k)=(\pi_1(A_k)\ast \pi_1(B_k))/N_k.
	$$
	By utilizing the isomorphism mentioned above, it can be concluded that: 
	
	\begin{theorem}
	\label{theorem:iso}
		The functors $\pi_1F_{X}$ and $(\pi_1F_A*\pi_1F_B)/N $ are isomorphic.
	\end{theorem}
	\begin{proof}
		Left to the reader. 
	\end{proof}
	
	
	
	We, of course, also have that $\pi_1G_{X'}$ and $(\pi_1G_{A'}*\pi_1G_{B'})/N' $ are isomorphic, where all of these terms are defined similarly.
	
	In the following, we show that if the persistent fundamental group functors $\pi_1F_A$ and $\pi_1G_{A'}$ are $\delta$-interleaved, if $\pi_1F_{B}$ and $\pi_1G_{B'}$ are $\delta$-interleaved, and if $\pi_1F_{A\cap B}$ and $\pi_1G_{A'\cap B'}$ are $\delta$-interleaved, then the persistent fundamental group functors $(\pi_1F_A*\pi_1F_B)/N $ and $(\pi_1G_{A'}*\pi_1G_{B'})/N' $ are also $\delta$-interleaved. 
	Using the isomorphisms from Theorem~\ref{theorem:iso}, this will then give that $\pi_1F_{X}$ and $\pi_1G_{X'}$ are $\delta$-interleaved.
	
		If the functors $\pi_1F_{A}$ and $\pi_1G_{A'}$ are $\delta$-interleaved, then for each $k\in \bR$ there exist two families of 
	homomorphisms $\{m(k) \colon \pi_{1}(A_k)\to \pi_{1}(A'_{k+\delta}) \}$ and $\{n(k)\colon \pi_{1}(A'_k)\to \pi_{1}(A_{k+\delta}) \}$ 
	such that the diagrams in Figure \ref{A_u2} commute.
	
	\begin{figure}[hbt]
	\centering
		\begin{tikzcd}            
			\pi_1(A_k) \arrow[rr , "\pi_1(A_k\ha A_l)" ] \arrow["m(k)" , d]    
			&&\pi_1(A_l) \arrow[d, "m(l)"] \\
			\pi_1(A'_{k+ \delta}) \arrow[rr,  "\pi_1(A'_{k+ \delta}\ha A'_{l+ \delta})" ]
			&&\pi_1(A'_{l+ \delta})
		\end{tikzcd}
		\begin{tikzcd}       
			\pi_1(A'_k) \arrow[rr , "\pi_1(A'_k\ha A'_l)" ] \arrow["n(k)" , d]    
			&&\pi_1(A'_l) \arrow[d, "n(l)"] \\
			\pi_1(A_{k+ \delta}) \arrow[rr , "\pi_1(A_{k+\delta}\ha A_{l+\delta})" ]
			&&\pi_1(A_{l+ \delta})
		\end{tikzcd}
		
  \vspace*{7mm}
		
		\begin{tikzcd}[column sep=.5em]
			& \pi_1(A'_{k+ \delta})  \arrow[dr,"n({k+\delta})"] \\
			\pi_1(A_{k})  \arrow[ur,"m(k)"] \arrow[rr
			, "\pi_1(A_k\ha A_{k+2\delta})"] && \pi_1(A_{k+ 2\delta}) 
		\end{tikzcd}	
		\begin{tikzcd}[column sep=.5em]
			& \pi_1(A_{k+ \delta})  \arrow[dr,"m({k+\delta})"] \\
			\pi_1(A'_{k})  \arrow[ur,"n(k)"] \arrow[rr,"\pi_1(A'_k\ha A'_{k+2\delta})"] && \pi_1(A'_{k+ 2\delta}) 
		\end{tikzcd}
		\caption{Commutative diagrams of $\delta$-interleaved functors $\pi_1F_A$ and $\pi_1G_{A'}$ (analogous diagrams exist for $\pi_1F_B$ and $\pi_1G_{B'}$).}
		\label{A_u2}
	\end{figure} 
	
	Also if the functors $\pi_1F_{B}$ and $\pi_1G_{B'}$ are $\delta$-interleaved, then there exists two families of homomorphisms 
	$\{s(k)\colon \pi_{1}(B_k)\to \pi_{1}(B'_{k+\delta}) \}$ and $\{u(k)\colon \pi_{1}(B'_{k})\to \pi_{1}(B_{k+\delta}) \}$ such that analagous diagrams to the ones in Figure \ref{A_u2} (with $A$ replaced by $B$, $n$ replaced by $s$, and $m$ replaced by $u$) commute.

 We also have similar families of interleving homomorphisms between $\pi_1((A\cap B)_k)$ and  $\pi_1((A'\cap B')_{k+\delta})$, and between $\pi_1((A'\cap B')_k)$ and $\pi_1((A\cap B)_{k+\delta})$.
 We suppress the notation for these homomorphisms to make our arguments easier to follow.

	For each $k\in \bR$, we can define a homomorphism  
	$$
	p(k) \colon (\pi_1(A_k)\ast \pi_1(B_k))/N_k \to (\pi_1(A'_{k+\delta})\ast \pi_1(B'_{k+\delta}))/N'_{k+\delta}
	$$ 
	\noindent 
	by  $p(k)((\Pi_{i=1}^n(a_i\cdot b_i))N_k)=(\Pi_{i=1}^n(m({k})(a_i)\cdot s({k})(b_i)))N'_{k+\delta}$, where the maps between $N_k$ and $N'_{k+\delta}$ are defined through the homomorphisms between $\pi_1((A\cap B)_k)$ and $\pi_1((A'\cap B')_{k+\delta})$.

	Similarly, for each $k\in \bR$, we can define another homomorphism 
	$$ 
	q(k) \colon (\pi_1(A'_{k})\ast \pi_1(B'_{k}))/N'_{k} \to (\pi_1(A_{k+\delta})\ast \pi_1(B_{k+\delta}))/N_{k+\delta}
	$$
	\noindent
	by $q(k)((\Pi_{i=1}^n(a'_i\cdot b'_i))N'_k)=(\Pi_{i=1}^n(n(k)(a'_i)\cdot u(k)(b'_i)))N_{k+\delta}$.

We remark that to have well-defined interleavings between the quotient spaces, we need to use the interleavings between intersection spaces.
The reason is that for $w\in \pi_1((A\cap B)_k)$, the elements $m(w)$ and $s(w)$ need not be equal to each other in $\pi_{1}(A'_{k+\delta})$ and $ \pi_{1}(B'_{k+\delta})$, but we do have that $m(w)\cdot s(w)^{-1}\in N'_{k+\delta}$.


		
	To prove that the functors $(\pi_1F_A*\pi_1F_B)/N$ and $(\pi_1G_{A'}*\pi_1G_{B'})/N'$ are $\delta$-interleaved, we need to show that the families of homomorphisms $p$ and $q$ form the commutative 
	diagrams in Figure \ref{A*B}.  We only show that the topmost diagram is commutative; the commutativity of the remaining diagrams can be checked similarly.
    \begin{figure}[ht]
		\centering
		\begin{tikzcd} 
			(\pi_1(A_k)*\pi_1(B_k))/ N_k \arrow[rrrr, "\pi_1 (A_k*B_k \to A_l*B_l)" ] \arrow["p(k)" , d]    
			&&&&(\pi_1(A_l)*\pi_1(B_l))/ N_l \arrow[d, "p(l)"] \\
			(\pi_1(A'_{k+\delta})*\pi_1(B'_{k+\delta}))/ N'_{k+\delta}  \arrow[rrrr, "\pi_1 (A'_{k+\delta}*B'_{k+\delta} \to A'_{l+\delta}*B'_{l+\delta})"]
			&&&& (\pi_1(A'_{l+\delta})*\pi_1(B'_{l+\delta}))/ N'_{l+\delta}
		\end{tikzcd}
		\vspace*{10mm}
		
		\begin{tikzcd}                               (\pi_1(A'_{k})*\pi_1(B'_{k}))/ N'_{k} \arrow[rrrr , "\pi_1 (A'_{k}*B'_{k} \to A'_{l}*B'_{l}))" ] \arrow["q(k)" , d]    
			&&&&(\pi_1(A'_{l})*\pi_1(B'_{l}))/ N'_{l} \arrow[d, "q(l)"] \\
			(\pi_1(A_{k+\delta})*\pi_1(B_{k+\delta}))/ N_{k+\delta} \arrow[rrrr , "\pi_1 (A_{k+\delta}*B_{k+\delta} \to A_{l+\delta}*B_{l+\delta})" ]
			&&&&(\pi_1(A_{l+\delta})*\pi_1(B_{l+\delta}))/ N_{l+\delta} 
		\end{tikzcd}
		\vspace*{10mm}
		
		\begin{tikzcd}[column sep=.3em]
			& (\pi_1(A'_{k+\delta})*\pi_1(B'_{k+\delta}))/ N'_{k+\delta}  \arrow[dr,"q({k+\delta})"] \\
			(\pi_1(A_{k})*\pi_1(B_{k}))/ N_{k}   \arrow[ur,"p(k)"] \arrow[rr
			, "\pi_1 (A_{k}*B_{k} \to A_{k+2\delta}*B_{k+2\delta})"] && (\pi_1(A_{k+2\delta})*\pi_1(B_{k+2\delta}))/ N_{k+2\delta} 
		\end{tikzcd}	
		\vspace*{10mm}
		
		\begin{tikzcd}[column sep=.3em]
			& (\pi_1(A_{k+\delta})*\pi_1(B_{k+\delta}))/ N_{k+\delta}  \arrow[dr,"p({k+\delta})"] \\
			(\pi_1(A'_{k})*\pi_1(B'_{k}))/ N'_{k}  \arrow[ur,"q(k)"] \arrow[rr
			, "\pi_1 (A'_{k}*B'_{k} \to A'_{k+2\delta}*B'_{k+2\delta})"] && (\pi_1(A'_{k+2\delta})*\pi_1(B'_{k+2\delta}))/ N'_{k+2\delta} 
		\end{tikzcd}
		\caption{$\delta$-interleaving for  the functors $(\pi_1F_A*\pi_1F_B)/N $ and $(\pi_1G_{A'}*\pi_1G_{B'})/N'$. }
		\label{A*B}
	\end{figure}
		
		
		

	For the topmost diagram in Figure \ref{A*B} to be commutative, we must check that
	\begin{equation}
    \label{eq:commutativity}
	p(l) \circ \pi_1 (A_k*B_k \to A_l*B_l) = \pi_1 (A'_{k+\delta}*B'_{k+\delta} \to A'_{l+\delta}*B'_{l+\delta}) \circ p(k).
	\end{equation}
	Let $(\Pi_{i=1}^n(a_i\cdot b_i))N_k\in (\pi_1(A_k)*\pi_1(B_k))/ N_k$.  Then 
	\begin{align}
		&p(l)\big(\pi_1 (A_k*B_k \to A_l*B_l)((\Pi_{i=1}^n(a_i\cdot b_i))N_k)\big)\nonumber\\
		=\,&p(l)\big( \Pi_{i=1}^n (\pi_1(A_k\ha A_l)(a_i)\cdot \pi_1(B_k\ha B_l)(b_i))N_l \big)\nonumber\\
		=\,&\big(\Pi_{i=1}^n (m(l)(\pi_1(A_k\ha A_l)(a_i)) \cdot s(l)(\pi_1(B_k\ha B_l)(b_i)))\big)N'_{l+\delta}\nonumber 
	\end{align}
	and 
	\begin{align}
		&\pi_1 (A'_{k+\delta}*B'_{k+\delta} \to A'_{l+\delta}*B'_{l+\delta})\big(p(k)((\Pi_{i=1}^n(a_i\cdot b_i))N_k)\big) \nonumber\\
		=\,&\pi_1 (A'_{k+\delta}*B'_{k+\delta} \to A'_{l+\delta}*B'_{l+\delta})\big((\Pi_{i=1}^n(m(k)(a_i)\cdot s(k)(b_i)))N'_{k+\delta}\big)\nonumber\\
		=\,&\big(\Pi_{i=1}^n(\pi_1(A'_{k+\delta}\ha A'_{l+\delta})(m(k)(a_i))\cdot \pi_1(B'_{k+\delta}\ha B'_{l+\delta})(s(k)(b_i)))\big)N'_{l+\delta}\nonumber 	
	\end{align}	
	
	We have both of the equalities $m(l)  \circ  \pi_1(A_k\ha A_l) = \pi_1(A'_{k+\delta}\ha A'_{l+\delta}) \circ m(k)$  and $s(l) \circ \pi_1(B_k\ha B_l) = \pi_1(B'_{k+\delta}\ha B'_{l+\delta}) \circ  s(k)$, giving equality in~\eqref{eq:commutativity}. 
	Thus, the topmost diagram in Figure \ref{A*B} is commutative.
	The rest follow similarly, and so we have the following theorem.

	\begin{theorem}
	\label{thm:interleaving}
		Let $X$ and $X'$ be based topological spaces decomposed as a union of path-connected open subsets $A$ and $B$ and 
		$A'$ and $B'$, respectively.
		If $\pi_1F_{A}$ and $\pi_1G_{A'}$ are $\delta$-interleaved, if $\pi_1F_{B}$ and $\pi_1G_{B'}$ are $\delta$-interleaved, and if $\pi_1F_{A\cap B}$ and $\pi_1G_{A'\cap B'}$ are $\delta$-interleaved, then 
		$(\pi_1F_A*\pi_1F_B)/N$ and $(\pi_1G_{A'}*\pi_1G_{B'})/N'$ are also $\delta$-interleaved.
	\end{theorem}

	
	Note that if $F$ and $G$ are $\delta$-interleaved, then they are also $\epsilon$-interleaved for any $\epsilon \geq \delta$. 
	Under the conditions of the above theorem, we have the following result.
	
	\begin{corollary}
	\label{cor:interleaving}
		The interleaving distance between the persistence fundamental group functors $\pi_1F_{X}$ and $\pi_1G_{X'}$  satisfies 
		$$
		d_I(\pi_1F_{X}, \pi_1G_{X'}) \leq \textrm{max}~ \{( d_I(\pi_1F_{A}, \pi_1G_{A'}), d_I(\pi_1F_{B}, \pi_1G_{B'}), d_I(\pi_1F_{A\cap B}, \pi_1G_{A'\cap B'}) \}. 
		$$		
	\end{corollary}

\section{Excision and Hurewicz Theorems for Persistent Homotopy Groups}
 \label{4}
	
	The excision property, which enables one to relate homology groups of a pair to that of a pair of subspaces, is one of the main reasons why homology can often be effectively calculated.   
	Homotopy groups, conversely, do not satisfy excision, which is why they are generally much harder to calculate.  However, depending on connectivity, there is a specific dimension range in which excision holds for homotopy groups (for definitions and homotopy theoretical properties, we again refer the reader to \cite{hatcher}).  This section shows that persistent homotopy benefits from the excision and Hurewicz theorems.  We also obtain the Freudenthal suspension theorem for persistent homotopy groups as a result of the excision theorem.

    Let $X$ be a topological space with a fixed basepoint $x_0 \in X$.
    We recall that $\pi_n(X,x_0)$ is the homotopy classes of maps $S^n\to X$ that fix the basepoint.
    Similarly, for $X$ a space with $x_0 \in A \subseteq X$, we recall that $\pi_n(X,A,x_0)$ is the homotopy classes of maps of pairs $(D^n,S^{n-1}) \to (X,A)$ that fix the basepoint.
	
    Let $F_{X}$ be a filtration for $X$, i.e., $F_{X}\in \mathbf{Top_{\bullet}^{R}} $ and 
	$x_0 \in X_k$ for all $k\in \bR$. Recall we define the \textbf{$(k, l)$-persistent homotopy group} of $X$ with respect to the filtration $F_{X}$, denoted $\im \pi_n(X_{kl})$, to be the image of the group homomorphism 
	$\pi_n(X_{kl})\colon \pi_n(X_k, x_0)\to \pi_n(X_l,x_0)$ induced by the inclusion of $X_k$ into $X_l$ (cf.~\cite{letscher}).  
	
	Let us also assume that $X$ is a CW complex decomposed as the union of subcomplexes $A$ and $B$ with the intersection $C=A\cap B$ connected and $x_0\in C$.
    As in the previous section, the filtration $F_{X}$
	induces  filtrations on $A$, $B$ and $C$ which we denote by $F_{A}$, $F_{B}$ and $F_{C}$, respectively.
    We assume that $x_0\in C_k$ for all $k\in \mathbb{R}$. 
    Moreover, $F_{X}$ induces filtrations for the relative pairs $(A, C)$, $(B,C)$ and $(X, B)$ which we denote by the functors $F_{(A, C)}$, $F_{(B, C)}$ and $F_{(X, B)}$, respectively, such that $F_{(A,C)}(k):={(A,C)}_k=(A_k,C_k)$, $F_{(B,C)}(k):={(B,C)}_k=(B_k,C_k)$ and $F_{(X,B)}(k):={(X,B)}_k=(X_k,B_k)$ for each $k\in \bR$.

	Similarly, we define the \textbf{$(k, l)$-persistent homotopy group} of a relative pair $(A, C)$ concerning a chosen filtration as the image of the group homomorphism $\pi_n({(A,C)}_{kl}): \pi_n({(A,C)}_k)\to \pi_n({(A,C)}_l)$ 
    induced by the inclusion of ${(A,C)}_k$ into ${(A,C)}_l$. 
    We denote this group by $\im \pi_n({(A,C)}_{kl})$.    

	Recall that a space $X$ is said to be \textbf{$n$-connected} if $\pi_k(X)=0$ for $k\leq n$.  Similarly, the pair $(X, A)$ is called \textbf{$n$-connected}                  if $\pi_k(X, A)=0$ for $k\leq n$.
	Let us also recall the excision theorem for homotopy groups.
	
    \begin{theorem}[Excision]
    \label{excision}
	If the pair $(A,C)$ is $m$-connected and the pair $(B,C)$ is $n$-connected, for $m,n \geq 0$, with $C$ connected and nonempty, then the map $\pi_u(A,C)\to \pi_u(X,B)$ 
	induced by inclusion is an isomorphism for $u< m+n$ and a surjection for $u=m+n$.
	\end{theorem}
		
	Next, we state and prove an excision theorem for persistent homotopy groups:  In the remainder of this section let us fix two filtration levels $k$ and $l$ with $k<l$ and assume that the spaces $A_k$, $B_k$, $X_k$, $A_l$, $B_l$, $X_l$ are subcomplexes of $X$.  
	Suppose that the relative pair ${(A, C)}_k$ is $m_1$-connected, ${(B, C)}_k$ is $n_1$-connected, ${(A, C)}_l$ is $m_2$-connected, and ${(B, C)}_l$ is $n_2$-connected.  
	Suppose also that $C_k$ and $C_l$ are non-empty and connected.  By Theorem \ref{excision}, we have  the following isomorphisms induced by inclusions:
	\begin{align*}
	f_k &\colon \pi_u({(A,C)}_k)  \to \pi_u({(X,B)}_k) &&\text{ for }u< m_1+n_1\text{, and} \\
	f_l &\colon \pi_u({(A,C)}_l) \to \pi_u({(X,B)}_l) &&\text{ for }u< m_2+n_2.
	\end{align*} 
	Note that both $f_k$ and $f_l$ are isomorphisms for $u< \textrm{min}  \{m_1+n_1, m_2+n_2\}$.
	By naturality, we have the following commutative diagram: 
	\begin{figure}[hbt]
		\centering
		\begin{tikzcd}                             
			\pi_u((A,C)_k) \arrow[rr , "\pi_u ({(A,C)}_{kl})" ] \arrow["f_k " , d]    
			&&\pi_u((A,C)_l) \arrow[d, "f_l"] \\
			\pi_u((X,B)_k) \arrow[rr , "\pi_u ({(X,B)}_{kl})" ]
			&&\pi_u((X,B)_l)
		\end{tikzcd}
		\caption{Commutative diagram between levels $k$ and $l$.}
		\label{A}
	\end{figure}

	We state the following \textit{excision theorem for persistent homotopy groups}, which, under certain connectivity conditions, allows us to work with a pair of smaller spaces,
	$(A, C)$, and get information about a pair of bigger spaces, $(X, B)$.

	\begin{theorem}\label{Exc}
	 For  fixed  filtration levels $k$ and $l$ with $k<l$, suppose that the relative pair ${(A, C)}_k$ is $m_1$-connected, ${(B, C)}_k$ is $n_1$-connected, 
         ${(A, C)}_l$ is $m_2$-connected, and ${(B, C)}_l$ is  $n_2$-connected.  Suppose also that $C_k$ and $C_l$ are non-empty and connected.
		Let $\alpha \colon \im \pi_u({(A,C)}_{kl}) \to \im \pi_u({(X,B)}_{kl})$ be the map defined by 
        $\alpha(a)=f_l(a)$.
		Then $\alpha$ is an isomorphism for $u< \textrm{min} \{m_1+n_1, m_2+n_2\}$ and a surjection for  $u=\textrm{min} \{m_1+n_1, m_2+n_2\}$.   
	\end{theorem}
	
	\begin{proof}
    Note that $\alpha$ is a homomorphism since $f_l$ is.
    By the commutativity of the diagram in Figure \ref{A}, the map $\alpha \colon \im \pi_u({(A,C)}_{kl}) \to \im \pi_u({(X,B)}_{kl})$ could equivalently be defined by 
		$
		\alpha(a)=\pi_u ({(X,B)}_{kl})(f_k(b))
		$
		\noindent 
		for any $b\in \pi_u({(A, C)}_k)$ satisfying $\pi_u ({(A,C)}_{kl})(b)=a$.
    
	We will prove that $\alpha$ is injective for $u< \textrm{min} \{m_1+n_1, m_2+n_2\}$ and surjective for $u\le \textrm{min} \{m_1+n_1, m_2+n_2\}$.

		\textbf{The map $\alpha$ is one-to-one:}  Let $a, a' \in \im \pi_u({(A,C)}_{kl})$ be such that $\alpha(a)=\alpha(a')$.  By definition, $a$ and $a'$ are nonzero elements in $ \pi_u({(A, C)}_l$ 
		such that $a=\pi_u ({(A,C)}_{kl})(b)$ and $a'=\pi_u ({(A,C)}_{kl})(b')$ for some nonzero $b$ and $b'$ in $\pi_u({(A, C)}_k)$, respectively.  Hence, by definition of $\alpha$, we have
		$$
		\pi_u ({(X,B)}_{kl})(f_k(b))=\pi_u ({(X,B)}_{kl})(f_k(b')).
		$$ 
		By commutativity of the diagram in Figure \ref{A}, we get 
		$$
		f_l(\pi_u ({(A,C)}_{kl})(b))=f_l(\pi_u ({(A,C)}_{kl})(b')).
		$$  
		Since $f_l$ is isomorphism for $u< \textrm{min} \{m_1+n_1, m_2+n_2\}$, we obtain $\pi_u ({(A,C)}_{kl})(b)=\pi_u ({(A,C)}_{kl})(b')$, which means $a=a'$.  Therefore, $\alpha$ is one-to-one.
		
		\textbf{The map $\alpha$ is onto:}  Let $c$ be a nonzero element of  $\im \pi_u({(X,B)}_{kl})$.  By definition of persistent homotopy groups, there exists a nonzero $c'\in \pi_u({(X, B)}_k)$ such that $c=\im \pi_u({(X,B)}_{kl})(c')$.  Since $f_k$ is surjective for $u\le \textrm{min} \{m_1+n_1, m_2+n_2\}$, there exists a nonzero 
		$b\in \pi_u({(A, C)}_k)$ such that $c'=f_k(b)$.  Let $a=\pi_u ({(A,C)}_{kl})(b)$, then by definition $a\in \im \pi_u({(A,C)}_{kl})$.  We have 
		$$
		\alpha(a)=\pi_u({(X,B)}_{kl})(f_k(b))=\pi_u({(X,B)}_{kl})(c')=c,
		$$ which means that $\alpha$ is onto.
	\end{proof}
In the following example, we show how to get information about a pair of spaces by working on a pair of smaller spaces. 
\begin{example} \label{torus-excision}
    Let $X$ be the torus decomposed as $X=A\cup B$, where $A\cap B$ is the circle $C$ as shown in Figure \ref{fig:torus}.  Let $f \colon X\to \bR$ be the height function whose sublevelsets are assembled into the filtrations $F_{(A, C)}$, $F_{(B, C)}$ and $F_{(X, B)}$ as above.
    Let $k<l$ be two fixed filtration levels.  Note that the pairs ${(A, C)}_k$ and ${(B, C)}_k$ are $m$-connected for any positive integer $m$.  The pair ${(A, C)}_l$ is $1$-connected and the pair ${(B, C)}_l$ is $0$-connected.
    By Theorem \ref{Exc}, we have an isomorphism
    $$\im \pi_u({(A,C)}_{kl})\cong \im \pi_u({(X,B)}_{kl})$$
    for $u<1$ and a surjection $\alpha \colon \im \pi_u({(A,C)}_{kl}) \to \im \pi_u({(X,B)}_{kl})$ for $u=1$.  Since $\im \pi_1({(A,C)}_{kl})$ is trivial, we can immediately conclude that $\im \pi_1({(X,B)}_{kl})$ is also trivial.

\begin{figure}[hbt]
			\centering
	\hspace{1.5 cm}	\includegraphics[width=.5\textwidth,height=.2\textheight]{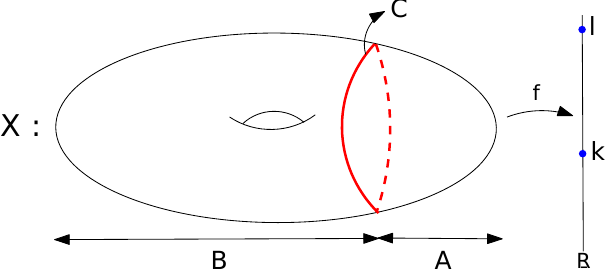}
			\caption{The height function $f$ on $X=A\cup B$.}
			\label{fig:torus}
		\end{figure}

Throughout the paper, and especially in Section~\ref{5}, we refer to such persistent homotopy groups as the \emph{sublevelset persistent homotopy groups of $f$}.
        
\end{example}


\begin{remark}
To motivate the use of excision, we point out that directly computing $\im \pi_1({(X,B)}_{kl})$ faces three key challenges:
\begin{itemize}
    \item[(i)] \text{Relative Loop Complexity}: Loops in \( X \) constrained to start/end in \( B \) require tracking both the topology of \( X \) and the subspace \( B \). For the torus, this involves non-contractible loops wrapping around handles, whose persistence depends on how \( B \) evolves across \( k \) and \( l \).
    
    \item[(ii)] \text{Filtration Dynamics}: As \( B \) changes with \( k \) (e.g., merging path components or forming new holes), the relations in \( \pi_1(X, B) \) may change. For instance, a loop that is non-trivial at \( k \) may become trivial at \( l \) if \( B \) fills in a hole.

     \item[(iii)] \text{Global Interactions}: The inclusion \( B \hookrightarrow X \) couples local features of \( B \) (e.g., its own fundamental group) with the global topology. Isolating contributions from \( B \) alone is infeasible without decomposition.
\end{itemize} 

In Example~\ref{torus-excision}, we note that
     \( A \) is chosen such that \( (A, C) \) is $1$-connected, simplifying $\im \pi_1({(A,C)}_{kl})$.
     The intersection \( C = A \cap B \) is a circle, whose persistent \( \pi_1 \)-groups are well-understood.
     Excision replaces $\im \pi_1({(X,B)}_{kl})$ with $\im \pi_1({(A,C)}_{kl})$, reducing the problem to a localized computation on \( A \).

Let us also mention impacts on computations:
\begin{itemize}
    \item[(i)] \text{Localization}: By focusing on \( A \) and \( C \), we avoid analyzing the full torus \( X \).
    For instance, if \(A\) is a contractible neighborhood of a critical point in the energy landscape (e.g., the vertex \(aaa\) in hexane's \(3\times3\times3\) grid at energy level \(3\beta\) in Figure~\ref{fig:Filled_cubes} (B), or other critical points analyzed in Section 5.2), \(\mathrm{Im}\pi_1((A,C)_{kl})\) becomes trivial.



    \item[(ii)] \text{Dimension Reduction}: The torus \( X \) has non-trivial \( \pi_1 \)-generators, but excision shifts focus to simpler subspaces where generators are easier to track.

 \item[(iii)]\text{Robustness of Persistence Intervals}: 
 Since $\pi_1$ is nonabelian, one cannot literally speak of a barcode as in the homology setting, but one can still record the birth and death values of each persistent class in $\mathrm{Im}\,\pi_1((X,B)_{k\ell})$.
 Excision guarantees that these birth–death parameters coincide with those of $\mathrm{Im}\,\pi_1((A,C)_{k\ell})$.
 In particular, any persistent loop in $(X,B)$ appears (and dies) exactly when its image appears (and dies) in $(A,C)$.
 Thus, the \emph{persistence intervals} of nontrivial $\pi_1$-elements are preserved under excision; in this sense, we refer to their \emph{robustness} under decomposition.
\end{itemize}

In practice, one often decomposes $X$ into subspaces whose $\pi_1$ is trivial or at least easily computed.  Excision then allows one to read off the persistent $\pi_1$ of $(X,B)$ from these simpler pieces, thereby circumventing the combinatorial and global difficulties of a direct calculation.  
\end{remark}

In conclusion,
Example~\ref{torus-excision} illustrates how excision circumvents the combinatorial and global challenges of computing persistent relative homotopy groups. By decomposing \( X \) into simpler subspaces \( A \) and \( B \), the theorem transforms an intractable problem into one that more feasible, both computationally and theoretically. This underscores excision as a useful tool for persistent homotopy calculations in complex spaces, such as energy landscapes.


As a corollary of Theorem \ref{Exc}, we prove a Freudenthal suspension theorem, which defines a stability property for persistent homotopy groups.  
	
	Consider the suspension $SX$ of $X$ as the union of two copies of the cone over $X$, that is, $SX=C_+X\cup C_-X$, where $C_+X= (X\times [0,1])/(X\times \{1\})$ and $C_-X=(X\times [-1,0])/(X\times \{-1\})$.  
	The filtration $F_{X}$ of $X$ induces filtrations on $SX$, $C_+X$ and $C_-X$ as $F_{SX}$, $F_{C_+X}$ and $F_{C_-X}$, respectively, such that $F_{SX}(k):=(SX)_k:=SX_k$, $F_{C_+X}(k):=(C_+X)_k:=C_+X_k$ and
	$F_{C_-X}(k):=(C_-X)_k:=C_-X_k$, where the suspensions and cones are taken at the point $x_0$.

	
	Next, we state some lemmas from which the suspension theorem for persistent homotopy groups follows as a corollary. 	
	\begin{lemma}\label{l1} 
		Suppose that $X_k$ is $(n-1)$-connected and $X_l$ is $(m-1)$-connected.  Then 
		$$
		\im \pi_{u+1}({(C_+X, X)}_{kl}) \cong \im\pi_{u+1}({(SX, C_-X)}_{kl})
		$$ 
		for $u+1< \textrm{min} \{2n, 2m\}$.
	\end{lemma}
	
	\begin{proof}
		Since $X_k$ is $(n-1)$-connected we have $\pi_{u}(X_k)=0$ for $u\leq n-1$, and since $C_+X_k$ is contractible we have $\pi_{u}(C_+X_k)=0$ for each $u>0$.  
		Thanks to the following long exact sequence of the relative pair ${(C_+X, X)}_k$,
		$$
		\displaystyle
		\xymatrix @-1pc 
		{ &{\cdots}\ar[r]  & \pi_{u}(X_k) \ar[r] & \pi_{u}(C_+X_k)\ar[r]&  \pi_{u}({(C_+X, X)}_k)\ar[r] & \pi_{u-1}(X_k)\ar[r] & {\cdots}}
		$$
		we get $\pi_{u}({(C_+X, X)}_k)=0$ for $u\leq n$, that is, the relative pair ${(C_+X, X)}_k$ is $n$-connected.  Similarly, the pair ${(C_-X, X)}_k$ is $n$-connected 
		and the pairs ${(C_+X, X)}_l$ and ${(C_-X, X)}_l$ are $m$-connected.  Since $SX_k=C_+X_k\cup C_-X_k$ and $C_+X_k\cap C_-X_k=X_k$, for each $k$, we get 
		$$
		\im \pi_{u+1}({(C_+X, X)}_{kl}) \cong \im\pi_{u+1}({(SX, C_-X)}_{kl})
		$$ 
		for $u+1< \textrm{min} \{2n, 2m\}$ by Theorem \ref{Exc}.
	\end{proof}
	
	\begin{lemma}\label{l2}  
		Suppose that $X_k$ and $X_l$ are connected.  Then 
		$$
		\im \pi_{u+1}({(C_+X, X)}_{kl})  \cong \im \pi_{u}(X_{kl})
		$$ 
		for each $u>0$.
	\end{lemma}
	
	\begin{proof}

\begin{figure}[hbt]
			\centering
			\begin{tikzcd}                                    
				\pi_{u+1}({(C_+X, X)}_k) \arrow[rrr , "\pi_{u+1}({(C_+X, X)}_{kl})" ] \arrow["\partial_{u,k}" , d]    
				&&&\pi_{u+1}({(C_+X, X)}_l) \arrow[d, "\partial_{u,l}"] \\
				\pi_{u}(X_k) \arrow[rrr , "\pi_{u}(X_{kl})" ]
				&&&\pi_{u}(X_l)
			\end{tikzcd}
			\caption{Commutative diagram of homotopy groups.}
			\label{B}
		\end{figure}
 
		Consider Figure~\ref{B}, which is commutative by naturality.		
		The boundary homomorphisms $\partial_{u,k}$ and $\partial_{u,l}$ are isomorphisms for each $u>0$, which follows from the long exact sequence of the corresponding relative pairs. 
		Let $\gamma \colon \im \pi_{u+1}({(C_+X, X)}_{kl})  \to \im \pi_u(X_{kl})$ be the map defined as $\gamma(a)=\partial_{u,l}(a)$, or equivalently, $\gamma(a)=\pi_{u}(X_{kl})(\partial_{u,k}(b))$ where $a=\pi_{u+1}({(C_+X, X)}_{kl})(b)$ for some $b\in \pi_{u+1}({(C_+X, X)}_k)$. The map $\gamma$ is a well-defined isomorphism 
		which follows from the commutativity of Figure \ref{B}.
	\end{proof}
	
	\begin{lemma}\label{l3}  
		Suppose that $X_k$ and $X_l$ are connected.  Then $$\im\pi_{u+1}((SX)_{kl}) \cong \pi_{u+1}({(SX, C_-X)}_{kl})$$ for each $u>0$.
	\end{lemma}
	
	\begin{proof}
		Consider the following long exact sequence for the relative pair $F_{(SX, C_-X)}(k)$:
		$$
		\displaystyle
		\xymatrix @-1pc 
		{ &{\cdots}\ar[r] & \pi_{u+1}(C_-X_k) \ar[r] & \pi_{u+1}(SX_k)\ar[r]&  \pi_{u+1}({(SX, C_-X)}_k)\ar[r] & \pi_{u}(C_-X_k)\ar[r] & {\cdots}}
		$$
		Since $C_-X_k$ is contractible, $\pi_{u}(C_-X_k)=0$ for each $u>0$.  Thus, we have isomorphisms 
		$$
		Sf_k:\pi_{u+1}(SX_k) \to \pi_{u+1}({(SX, C_-X)}_k)
		$$ 
		and 
		$$
		Sf_l:\pi_{u+1}(SX_l) \to \pi_{u+1}({(SX, C_-X)}_l).
		$$
		Now consider the following commutative diagram:
		\begin{figure}[hbt]
			\centering
			\begin{tikzcd}                                    
				\pi_{u+1}(SX_k) \arrow[rrr , "\pi_{u+1}((SX)_{kl})" ] \arrow["Sf_k" , d]    
				&&&\pi_{u+1}(SX_l) \arrow[d, "Sf_l"] \\
				\pi_{u+1}({(SX, C_-X)}_k) \arrow[rrr , "\pi_{u+1}({(SX, C_-X)}_{kl})" ]
				&&&\pi_{u+1}({(SX, C_-X)}_l)
			\end{tikzcd}
			\caption{Commutative diagram of homotopy groups.}
			\label{C}
		\end{figure}
		
		Let $\beta \colon \im \pi_{u+1}((SX)_{kl}) \to \im\pi_{u+1}({(SX, C_-X)}_{kl})$ be the map defined by $\beta(a)=Sf_l(a)$, or equivalently, $$\beta(a)=\pi_{u+1}({(SX, C_-X)}_{kl})(Sf_k(b)),$$ where $a=\pi_{u+1}((SX)_{kl})(b)$ 
		for some $b\in \pi_{u+1}(SX_k)$.
    Using the commutativity of Figure \ref{C}, one can deduce that $\beta$ is an isomorphism.
	\end{proof}
	
	Now, \textit{the Freudenthal suspension theorem for persistent homotopy groups} follows as a corollary of the above lemmas.	
	\begin{theorem} \label{suspension}
		Suppose that $X_k$ is $(n-1)$-connected and $X_l$ is $(m-1)$-connected.  Then 
		$$
		\im \pi_{u}(X_{kl}) \cong \im \pi_{u+1}((SX)_{kl})
		$$ 
		for $u+1< \textrm{min} \{2n, 2m\}$. 
	\end{theorem}
	
	\begin{proof}
		The proof comes as a result of the Lemmas \ref{l1}, \ref{l2}, and \ref{l3}.
	\end{proof}
	
	Next, we state \textit{the Hurewicz Theorem for persistent homotopy groups}.  Consider the homology group functor $H_{n}:\mathbf{Top_{\bullet}} \to \mathbf{Gp}$.  We define the \textbf{$(k, l)$-persistent homology group} of $X$ with respect to the filtration $F_{X}$ to be the image of the group homomorphism 
	$H_n(X_{kl}): H_n(X_k)\to H_n(X_l)$ induced by the inclusion of $X_k$ into $X_l$.  Let us denote this group by $\im H_n(X_{kl})$.
	
In \cite[Theorem 1.5]{memoli-zhou}, the authors prove a persistent version of the Hurewicz theorem for persistent fundamental groups.
The following theorem tells us that, under certain connectivity conditions, there is a bijection between $(k, l)$-persistent homology classes and $(k, l)$-persistent homotopy classes.  
 
	\begin{theorem}
 \label{hurewicz}
		Suppose that $X_k$ is $(m-1)$-connected and $X_l$ is $(n-1)$-connected for $m, n\geq 2$.  
		Then $\im H_u(X_{kl})=0$ for $0<u< \textrm{min} \{m,n\}$, and 
		$$
		\im \pi_{u}(X_{kl}) \cong \im H_{u}(X_{kl})
		$$ for $u= \textrm{min} \{m,n\}$.
	\end{theorem}

	\begin{proof}
		If $X_k$ is $(m-1)$-connected, then the relative homology satisfies $\tilde{H_u}(X_k)=0$ for $u<m$ by the Hurewicz Theorem \cite{hatcher}.
        Since $H_u(X_k)=\tilde{H_u}(X_k)$, for $u>0$, then $H_u(X_k)=0$, for $0<u<m$. Thus, we obtain $\im H_u(X_{kl})=0$, for $0<u< \textrm{min} \{m,n\}$. 
		
		Note that $X_k$ and $X_l$ are $(u-1)$-connected for $u= \textrm{min} \{m,n\}$. Thus, by the Hurewicz Theorem, we have the following isomorphisms 
		$$
		h_k \colon \pi_{u}(X_k) \to H_{u}(X_k).
		$$ 
		and 
		$$
		h_l \colon \pi_{u}(X_l) \to H_{u}(X_l).
		$$
		Now, consider the following commutative diagram:
		
		\begin{figure}[hbt]
			\centering
			\begin{tikzcd}                                    
				\pi_{u}(X_k)\arrow[rr , "\pi_{u}(X_{kl})" ] \arrow["h_k" , d]    
				&&\pi_{u}(X_l)\arrow[d, "h_l"] \\
				H_{u}(X_k)\arrow[rr , "H_{u}(X_{kl})" ]
				&&H_{u}(X_l)) 
			\end{tikzcd}
			\caption{Commutative diagram between levels $k$ and $l$.}
			\label{D}
		\end{figure}
		Let $h \colon \im \pi_{u}(X_{kl}) \to \im H_{u}(X_{kl})$ be the map defined by $h(a)=h_l(a)$, or equivalently, as $h(a)=H_{u}(X_{kl})(h_k(b))$ where $a=\pi_{u}(X_{kl})(b)$ for some $b\in \pi_{u}(X_k)$. 
		Clearly, $h$ is a well-defined homomorphism.
       By using the commutative diagram given in Figure \ref{D}, one can easily obtain that $h$ is an isomorphism for $u= \textrm{min} \{m,n\}$.		
	\end{proof} 


\section{An Application to Energy Landscapes of Molecules}
\label{5}

 In this section, we analyze the sublevelset persistent homotopy groups of the energy landscape of alkane molecules.
 In particular, we explain the additional information these persistent homotopy groups contain beyond what is shown in the persistent homology barcodes.

\subsection{The Potential Energy Landscape of Alkanes}

An n-alkane molecule consists of a linear chain of carbon atoms, with three hydrogen atoms attached to the two carbons at the end of the chain and two hydrogen atoms attached to each internal carbon.
The n-alkane molecules with four, five, six, seven, and eight carbon atoms in the chain are called butane, pentane, hexane, heptane, and octane, respectively.

\begin{tikzcd}
            & H \arrow[d, no head]           & H \arrow[d, no head]           & H \arrow[d, no head]           &                  & H \arrow[d, no head]           & H \arrow[d, no head]           & H \arrow[d, no head]           &   \\
H \arrow[r, no head] & C \arrow[d, no head] \arrow[r, no head] & C \arrow[d, no head] \arrow[r, no head] & C \arrow[d, no head] \arrow[r, no head] & \ldots \arrow[r, no head] & C \arrow[r, no head] \arrow[d, no head] & C \arrow[r, no head] \arrow[d, no head] & C \arrow[d] \arrow[r, no head] & H \\
            & H                     & H                     & H                     &                  & H                     & H                     & H                     &  
\end{tikzcd}

We study the Optimized Potentials for Liquid Simulations (OPLS-UA) model~\cite{jorgensen1996development} for the energy function of alkanes, as considered in \cite{Mirth2021}.
In this OPLS-UA model, the potential energy landscape of an alkane molecule is governed exclusively by the C--C--C--C dihedral angles $\phi_i$.
In the case of butane, a carbon chain of length four, there is only a single dihedral angle $\phi\in S^1$, where $S^1$ is the circle; see Figure~\ref{fig:dihedral-angle}.
So the energy landscape of butane is a function $f_1 \colon S^1 \to \bR$ that is defined by
\[f_1(\phi)=c_1(1+\cos\phi)+c_2(1-\cos2\phi)+c_3(1+\cos3\phi).\]
Here the energy coefficients are $c_1/k_B=355.03$~K, $c_2/k_B=-68.19$~K, and $c_3/k_B=791.32$~K, where $k_B$ is the Boltzmann constant \cite{chen1998thermodynamic}.
An alkane molecule with $m$ carbon atoms has $n=m-3$ dihedral angles, and the corresponding OPLS-UA energy function $f_n \colon (S^1)^n\to\bR$ is defined by $f_n(\phi_1,\phi_2,\ldots,\phi_n)=f_1(\phi_1)+f_1(\phi_2)+\ldots+f_1(\phi_n)$, where each $\phi_i\in S^1$ encodes a different dihedral angle.

\begin{figure}[ht]
\includegraphics[width=0.4\textwidth]{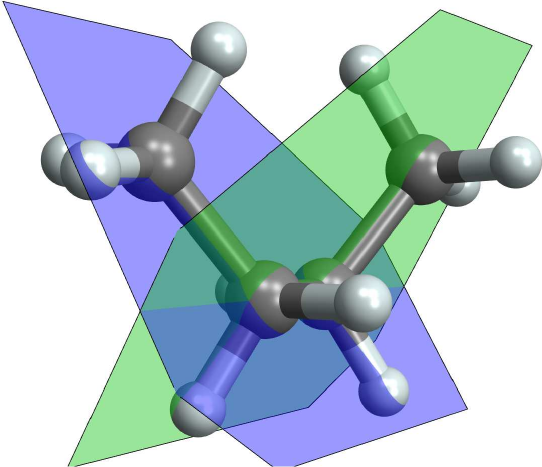}
\caption{Three adjacent carbons in an alkane chain define a plane, one drawn above in blue and another in green.
A C--C--C--C dihedral angle is the angle between two planes sharing a C--C bond.
In butane, drawn above, there is only a single dihedral angle.
Figure made in Mathematica.
}
\label{fig:dihedral-angle}
\end{figure}

\begin{figure}[ht]
\includegraphics[width=0.4\textwidth]{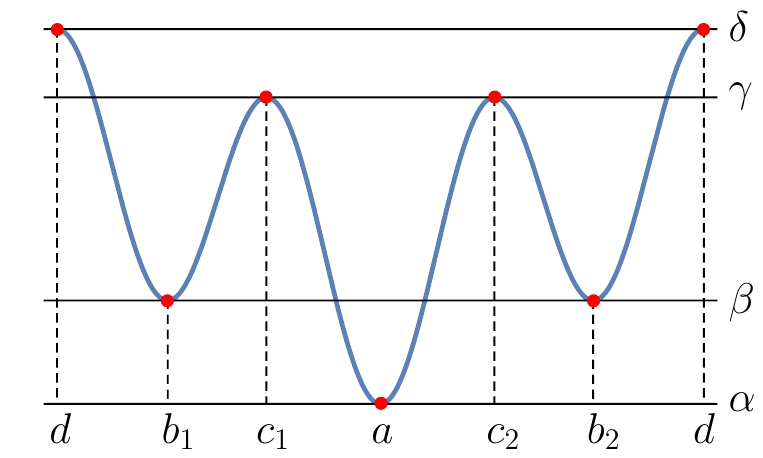}
\hspace{10mm}
\includegraphics[width=0.4\textwidth]{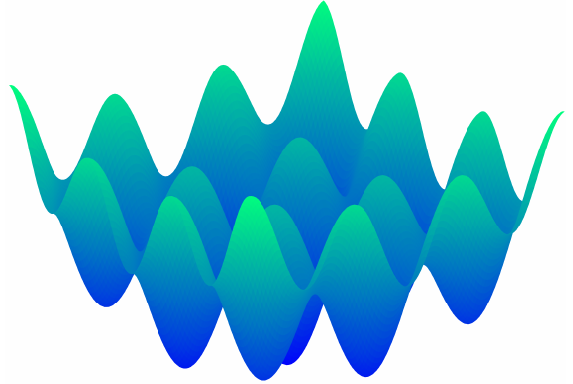}
\caption{
(Left) Energy landscape of butane.
The $y$-axis is energy, and the $x$-axis is the dihedral angle varying from $0$ to $2\pi$.
From left to right, the critical points are $d, b_1, c_1, a, c_2, b_2$.
(Right) Energy landscape of pentane.
Figures from \cite{Mirth2021}.
}
\label{fig:butane-labeled}
\end{figure}

See Figure~\ref{fig:butane-labeled}(left) for a picture of the energy landscape $f_1\colon S^1\to \bR$ for butane.
This Morse function has a global minimum $a$ with energy value $f_1(a)=\alpha=0$, two local minima $b_1$ and $b_2$ with energy value $f_1(b_1)=f_1(b_2)=\beta=3.47099\ldots$, two local maxima $c_1$ and $c_2$ with energy value $f_1(c_1)=f_1(c_2)=\gamma=13.8062\ldots$, and finally one global maximum $d$ with energy value $f_1(d)=\delta=19.0626\ldots$.
For pentane, the energy function $f_2\colon (S^1)^2\to \bR$ is shown in Figure~\ref{fig:butane-labeled}(right).
The sublevelsets of the pentane energy function $f_2$ are shown in Figure~\ref{fig:pentane-sublevelsets}.

\begin{figure*}
\centering
\includegraphics[width=0.44\textwidth]{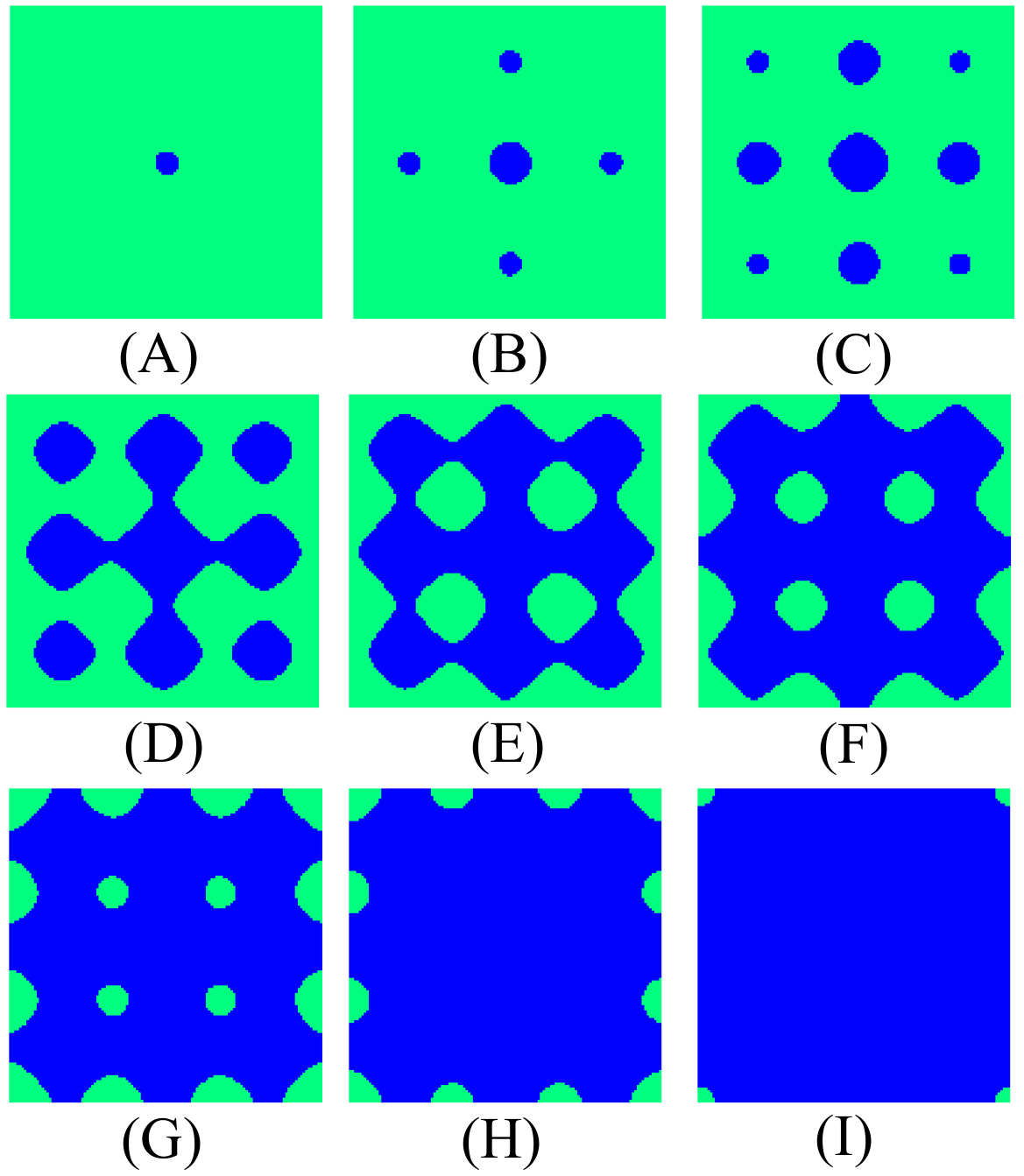}
\hspace{5mm}
\includegraphics[width=0.50\textwidth]{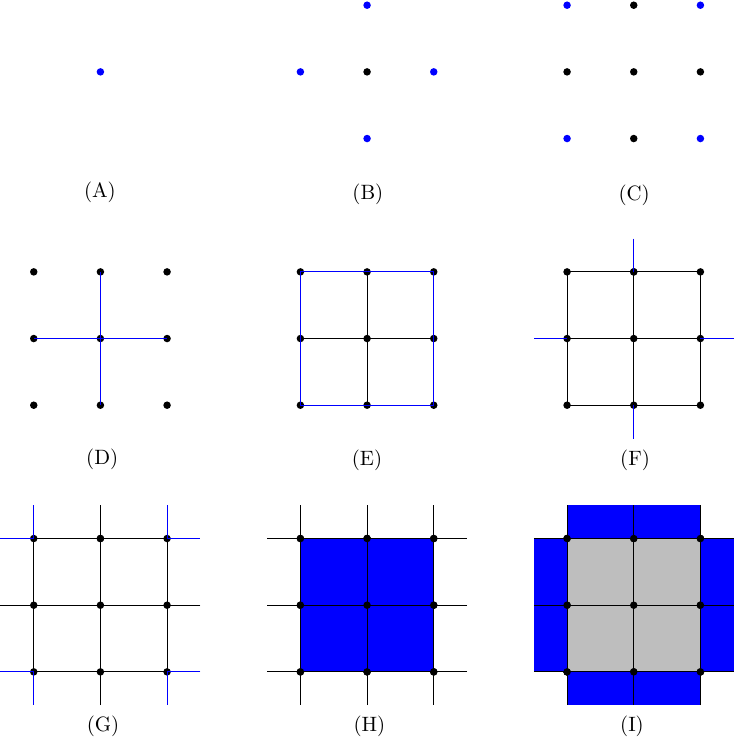}
\caption{
(Left) Pentane sublevelsets $f_2^{-1}(-\infty,r]:=\{y\in (S^1)^2~|~f_2(y)\le r\}$ are drawn, in blue, for increasing values of energy value $r$.
Figure from~\cite{Mirth2021}.
(Right) Pentane sublevelsets, drawn as subcomplexes of the Morse complex.
The entire Morse complex (not drawn) is a CW complex model of the torus $(S^1)^2$ with $9$ critical $0$-cells, $18$ critical $1$-cells, and $9$ critical $2$-cells.
}
\label{fig:pentane-sublevelsets}
\end{figure*}

A common low-dimensional representation of a high-dimensional energy landscape is via a \emph{merge tree} or \emph{disconnectivity graph}~\cite{wales2003energy,Wales2005}.
In such a representation, each connected component of an energy sublevelset corresponds to a vertex in a graph.
As a result, merge trees encode how new configurations and new transition paths between configurations emerge as the energy level increases.
However, merge trees do not contain any information about the \emph{shape} of each connected component of an energy landscape.
The paper~\cite{Mirth2021} studied the persistent homology of the energy sublevelsets of alkane chains, showing that connected components of the energy landscape can have complicated topologies with a large number of $i$-dimensional holes for $i\ge 1$.
Furthermore, an analytical formula was given for the OPLS-UA energy function $f_n \colon (S^1)^n\to\bR$ modelling $n+3$ carbons in an alkane chain, based on the K\"unneth formula for persistent homology~\cite{bubenik2021homological,carlsson2020persistent,gakhar2019kunneth,polterovich2017persistence}.
This work is extended in~\cite{story2023additive}, which shows how the persistent K\"unneth formula can be used to describe the persistent homology barcodes of the sublevelset persistence of any additive energy function over a product space, including, for example, branched alkanes.

\begin{figure*}
\centering
\includegraphics[width=.48\textwidth]{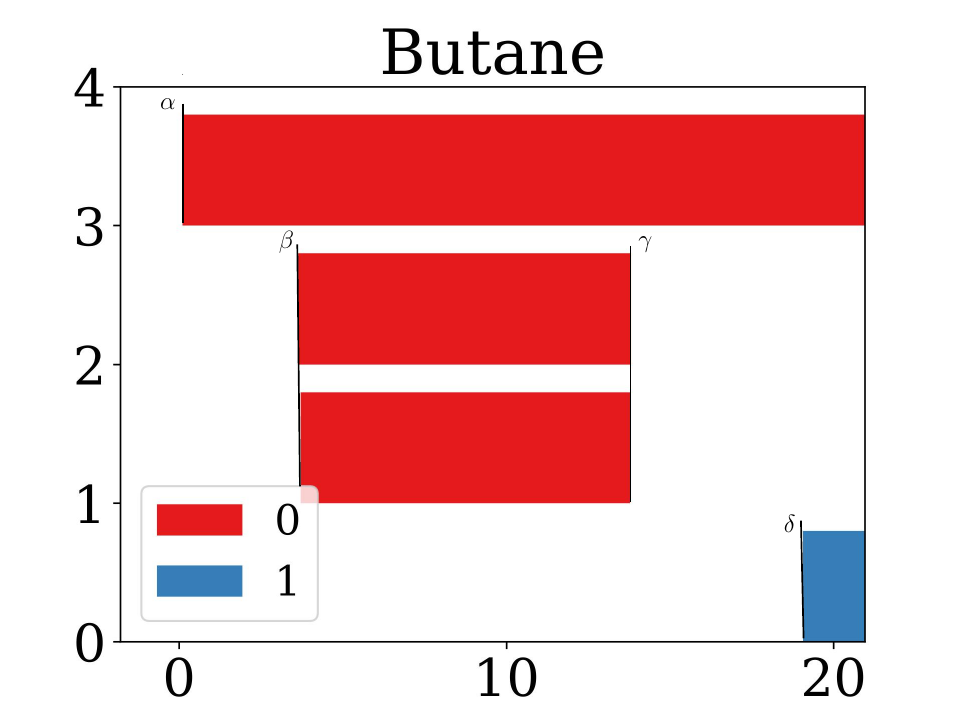}
\includegraphics[width=.48\textwidth]{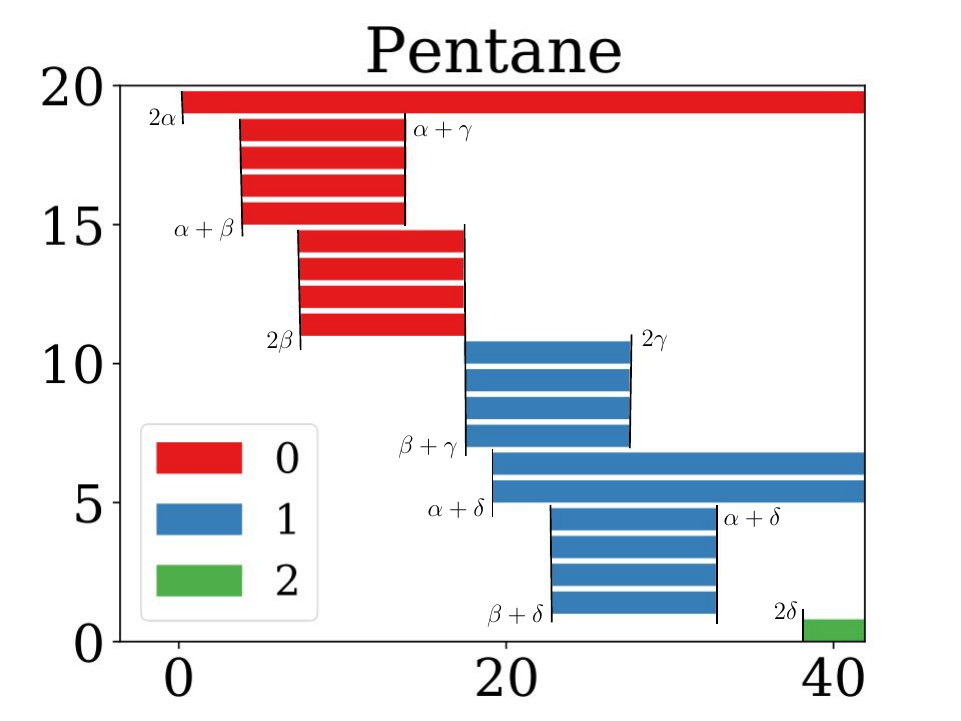}
\caption{
Figure from \cite{Mirth2021} with critical levels indicated.
Persistent homology barcodes for butane and pentane.
The 0-, 1-, and 2-dimensional homology features are shown in red, blue, and green, respectively.
The $x$-axis is energy (kJ/mol).}
\label{fig:alkane-barcodes}
\end{figure*}

\begin{figure*}
  \centering
\includegraphics[width=.6\textwidth]{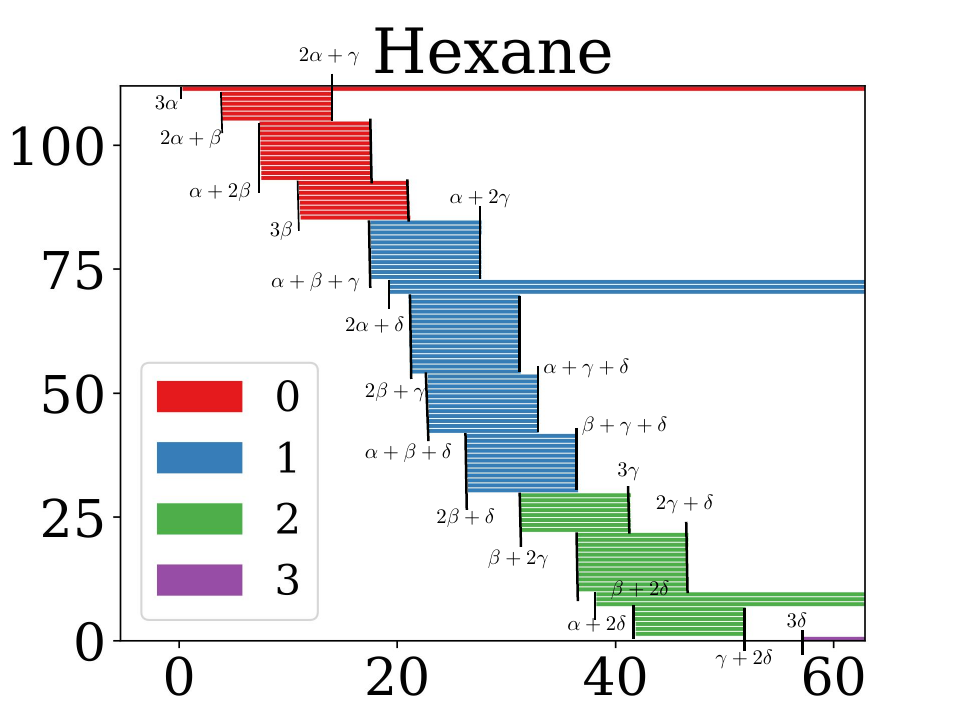}  
\caption{Figure from \cite{Mirth2021} with critical levels indicated.
Persistent homology barcodes for hexane.
The color of the bar indicates the homological dimension.
The $x$-axis is energy (kJ/mol).
}
\label{fig:hexane-barcodes}
\end{figure*}


\subsection{Analysis of the Persistent Homotopy Groups}

In this subsection, we study sublevelset persistent homotopy groups of the energy functions for butane, pentane and hexane.
For butane, there is essentially no difference between persistent homology and homotopy; for pentane, the only difference arises for the semi-infinite homology bars; and for hexane, significant differences between persistent homology and homotopy arise already for the finite-length persistent homology bars.


The butane and pentane sublevelset persistent homology barcodes are depicted in Figure~\ref{fig:alkane-barcodes}.
For butane, there is no significant difference between homology and homotopy persistence.
For pentane, the difference appears at the last critical energy level, $2\delta$, when the sublevelset becomes the entire torus $S^1 \times S^1$.
At the energy level $2\delta$, an $H_2$ persistent homology generator is introduced, and there is no change to $H_1$ (the two current bars continue).
On the other hand, when we consider persistent homotopy, at the last critical energy level $2\delta$, there is no new $\pi_2$ generator, but instead, this level introduces a commutator relation on the fundamental group.
Indeed, the fundamental group transitions from the free group on two generators to the free abelian group on two generators.

We now analyze the sublevelset persistent homotopy groups of the energy function $f_3\colon (S^1)^3\to \bR$ for hexane.
For the persistent homology barcodes, see Figure~\ref{fig:hexane-barcodes}.
There are $20$ critical sublevelsets for hexane, whose energies are ordered from smallest to largest as follows:
\begin{align*}
& 3\alpha, 2\alpha+\beta, \alpha+2\beta, 3\beta, 2\alpha + \gamma, \alpha + \beta + \gamma, 2\alpha+ \delta, \\
& 2\beta + \gamma, \alpha+\beta+\delta, 2\beta+\delta, \alpha+2\gamma, \beta+2\gamma, \alpha+\gamma+\delta,  \\
& \beta+\gamma+\delta, \alpha+2\delta, 3\gamma, \beta+2\delta, 2\gamma+\delta, \gamma+2\delta, 3\delta.
\end{align*}



As we increase the energy level past a critical value, a collection of critical points (with that prescribed critical energy value) appear in the sublevelset.
Each critical point of index $i$ changes the shape of the sublevelset, up to homotopy equivalence, by attaching a cell of dimension $i$~\cite{milnor2016morse, BanyagaHurtubise}.
As described in Appendices~A and B of~\cite{Mirth2021}, in this way, we build up the Morse complex of the energy function, which is a CW complex model of the $3$-dimensional torus $(S^1)^3$, containing $27$ critical $0$-cells, $81$ critical $1$-cells, $81$ critical $2$-cells, and $27$ critical $3$-cells.
We now describe these sublevelsets and their homotopy groups $\pi_i$ for $i\le 2$ as the energy values increase and hence as more critical cells are included.

In Table~\ref{table:hexane-homotopy-groups}, we list the fundamental and second homotopy groups of the sublevelsets of hexane, starting with the last energy value for which both $\pi_1$ and $\pi_2$ are the trivial groups.
By $\ast_k \bZ$, we denote the free group on $k$ generators, which is not abelian for $k\ge 2$, and by $\oplus_k \bZ$, we denote the free abelian group on $k$ generators.
This analysis does not consider the homotopy groups $\pi_i$ for $i\ge 3$, which can be quite complicated.
Indeed, even for the $2$-sphere we have $\pi_3(S^2)\cong \bZ\neq 0$, as generated by the Hopf fibration.

\begin{table}
\begin{align*}
& 2\alpha+\gamma: && \pi_1 \cong 0 && \pi_2 \cong 0\\
& \alpha+\beta+\gamma: && \pi_1 \cong \ast_{12} \bZ && \pi_2 \cong 0\\
& 2\alpha + \delta: && \pi_1 \cong \ast_{12} \bZ \ast_{3} \bZ && \pi_2 \cong 0\\
& 2\beta + \gamma:   && \pi_1 \cong \ast_{12} \bZ \ast_{3} \bZ \ast_{16} \bZ  && \pi_2 \cong 0\\
& \alpha + \beta + \delta:  &&\pi_1 \cong \ast_{12} \bZ \ast_{3} \bZ \ast_{16} \bZ \ast_{12}\bZ  && \pi_2 \cong 0\\ 
& 2\beta + \delta:  &&\pi_1 \cong \ast_{12} \bZ \ast_{3} \bZ \ast_{16} \bZ \ast_{12}\bZ \ast_{12}\bZ  && \pi_2 \cong 0\\
& \alpha + 2\gamma :  &&\pi_1 \cong \ast_{3} \bZ \ast_{16} \bZ \ast_{12}\bZ \ast_{12}\bZ  && \pi_2 \cong 0\\
& \beta + 2\gamma :  &&\pi_1 \cong \ast_{3} \bZ \ast_{12}\bZ \ast_{12}\bZ  && \pi_2 \cong \pi_2((\vee_{27}S^1)\vee(\vee_{8}S^2)) \cong \oplus_\infty \bZ\\
& \alpha + \gamma + \delta:  &&\pi_1 \cong \ast_{3} \bZ \ast_{12}\bZ  && \pi_2\cong \pi_2((\vee_{15}S^1)\vee(\vee_{8}S^2)) \cong \oplus_\infty \bZ\\
& \beta + \gamma + \delta:  &&\pi_1 \cong \ast_{3} \bZ     && 
 \pi_2\cong \pi_2((\vee_{3}S^1)\vee(\vee_{20}S^2)) \cong \oplus_\infty \bZ  \\
& \alpha + 2 \delta:  &&\pi_1 \cong  \oplus_{3}\bZ && \pi_2 \cong \pi_2(((S^1)^3\setminus\{p\})\vee(\vee_{20}S^2)) \\ 
& 3\gamma:  &&\pi_1 \cong \oplus_{3}\bZ && \pi_2 \cong \pi_2(((S^1)^3\setminus\{p\})\vee(\vee_{12}S^2)) \\
& \beta + 2 \delta :  &&\pi_1 \cong \oplus_{3}\bZ && \pi_2 \cong  \pi_2(
((S^1)^3\setminus\{p\})\vee(\vee_{18}S^2))  \\
& 2\gamma + \delta :  &&\pi_1 \cong  \oplus_{3}\bZ && \pi_2 \cong \pi_2(((S^1)^3\setminus\{p\})\vee(\vee_{6}S^2))   \\
& \gamma + 2 \delta  :  &&\pi_1 \cong \oplus_{3}\bZ  && \pi_2 \cong \pi_2((S^1)^3\setminus\{p\}) \\
& 3\delta :  &&\pi_1 \cong \oplus_{3}\bZ  && \pi_2 \cong 0  \\ 
\end{align*}
\caption{The fundamental and second homotopy groups of the sublevelsets of hexane, starting with the last energy value for which both $\pi_1$ and $\pi_2$ are the trivial groups.
By $\ast_k \bZ$ we denote the free group on $k$ generators, and by $\oplus_k \bZ$ we denote the free abelian group on $k$ generators.
For simplicity, at each row, $\pi_1$ and $\pi_2$ denote the homotopy groups of the sublevelset at the corresponding energy value.}
\label{table:hexane-homotopy-groups}
\end{table}

\begin{figure*}[htb]
\centering
\begin{subfigure}{0.23\textwidth}
\includegraphics[width=\textwidth]{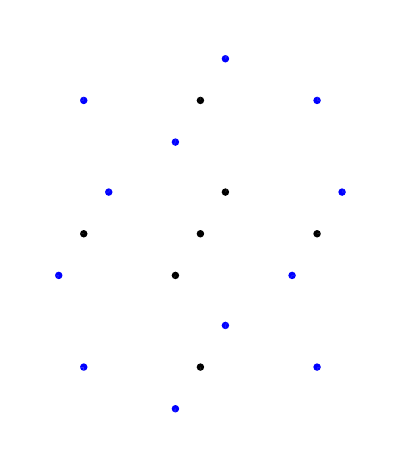}
\caption{$\alpha+2\beta$}
\label{fig:image_-3}
\end{subfigure}
\hfill
\begin{subfigure}{0.23\textwidth}
\includegraphics[width=\textwidth]{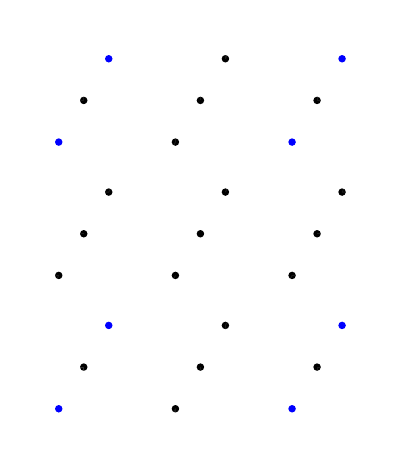}
\caption{$3\beta$}
\label{fig:image_-2}
\end{subfigure}
\hfill
\begin{subfigure}{0.23\textwidth}
\includegraphics[width=\textwidth]{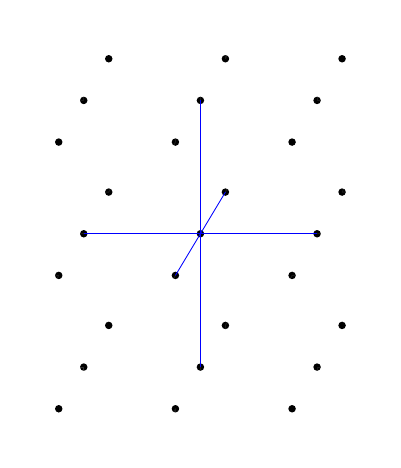}
\caption{$2\alpha+\gamma$}
\label{fig:image_-1}
\end{subfigure}
\hfill
\begin{subfigure}{0.23\textwidth}
\includegraphics[width=\textwidth]{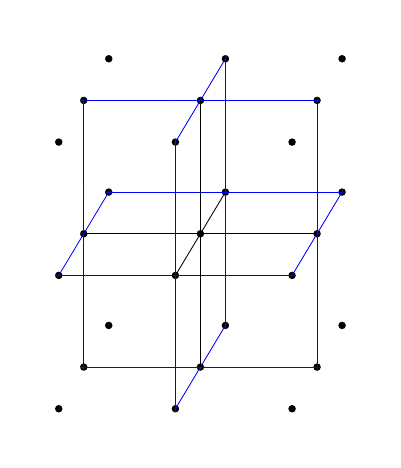}
    \caption{$\alpha+\beta+\gamma$}
    \label{fig:image_0}
\end{subfigure}
\hfill
\begin{subfigure}{0.23\textwidth}
\includegraphics[width=\textwidth]{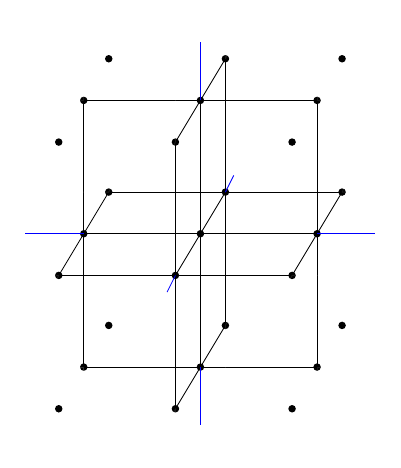}
\caption{$2\alpha+\delta$}
\label{fig:image_1}
\end{subfigure}
\hfill
\begin{subfigure}{0.23\textwidth}
\includegraphics[width=\textwidth]{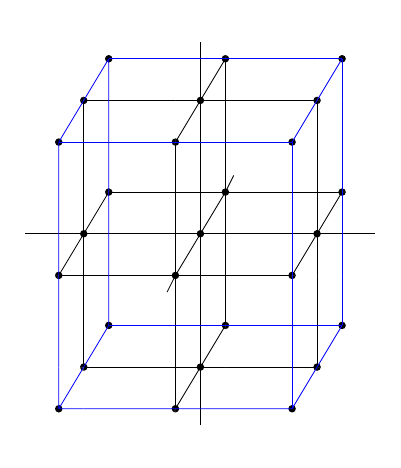}
\caption{$2\beta+\gamma$}
\label{fig:image_2}
\end{subfigure}
\hfill
\begin{subfigure}{0.23\textwidth}
\includegraphics[width=\textwidth]{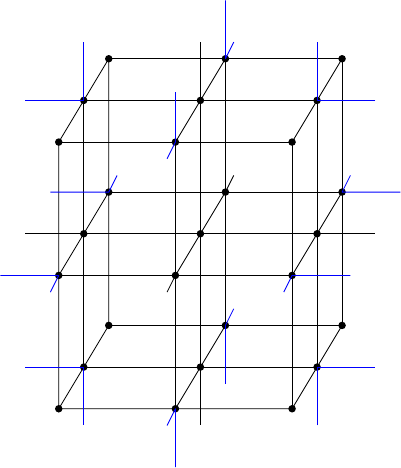}
\caption{$\alpha+\beta+\delta$}
\label{fig:image_3}
\end{subfigure}
\hfill
\begin{subfigure}{0.23\textwidth}
\includegraphics[width=\textwidth]{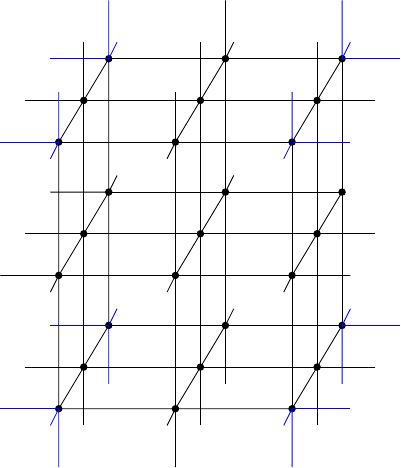}
    \caption{$2\beta+\delta$}
    \label{fig:image_4}
\end{subfigure}
\hfill
\begin{subfigure}{0.23\textwidth}
\includegraphics[width=\textwidth]{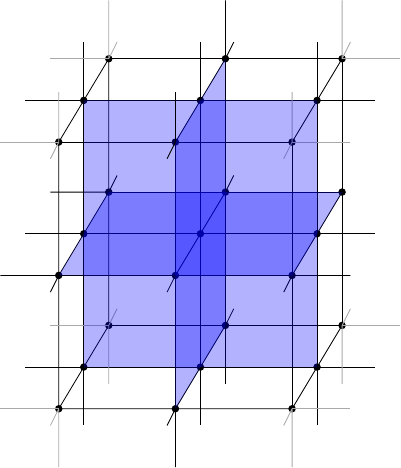}
\caption{$\alpha+2\gamma$}
\label{fig:image_5}
\end{subfigure}  
\hfill
\begin{subfigure}{0.23\textwidth}
\includegraphics[width=\textwidth]{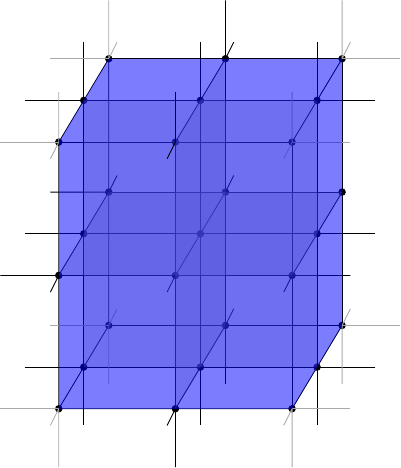}
\caption{$\beta+2\gamma$}
\label{fig:image_6}
\end{subfigure}
\hfill
\begin{subfigure}{0.23\textwidth}
\includegraphics[width=\textwidth]{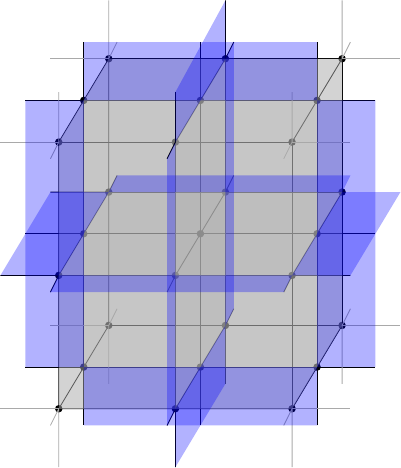}
    \caption{$\alpha+\gamma+\delta$}
    \label{fig:image_7}
\end{subfigure}
\hfill
\begin{subfigure}{0.23\textwidth}
\includegraphics[width=\textwidth]{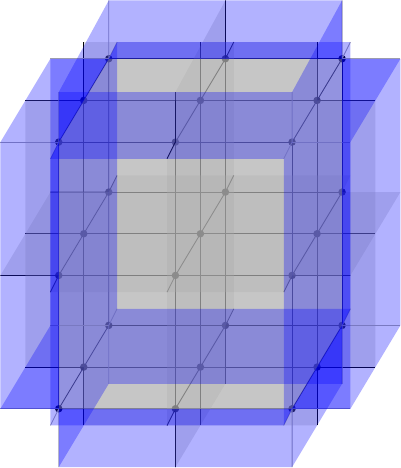}
\caption{$\beta+\gamma+\delta$}
\label{fig:image_8}
\end{subfigure}  
\hfill
\begin{subfigure}{0.23\textwidth}
\includegraphics[width=\textwidth]{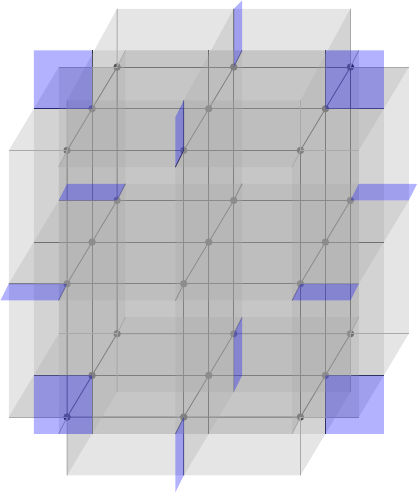}
\caption{$\alpha+2\delta$}
\label{fig:image_9}
\end{subfigure}
\hfill
\begin{subfigure}{0.23\textwidth}
\includegraphics[width=\textwidth]{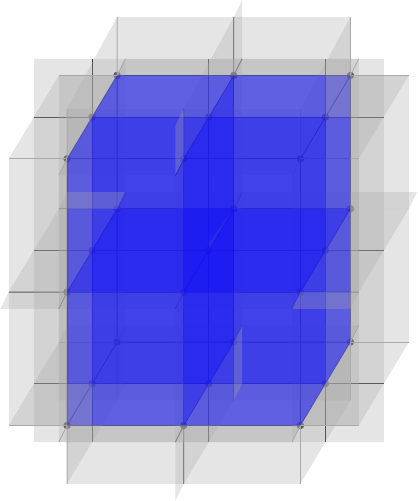}
\caption{$3\gamma$}
\label{fig:image_10}
\end{subfigure}
\hfill
\begin{subfigure}{0.23\textwidth}
\includegraphics[width=\textwidth]{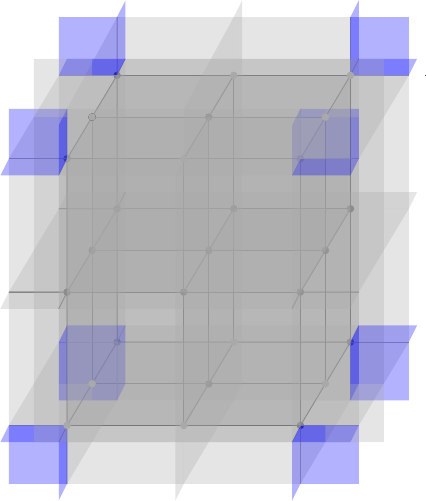}
\caption{$\beta+2\delta$}
\label{fig:image_11}
\end{subfigure}  
\hfill
\begin{subfigure}{0.23\textwidth}
\includegraphics[width=\textwidth]{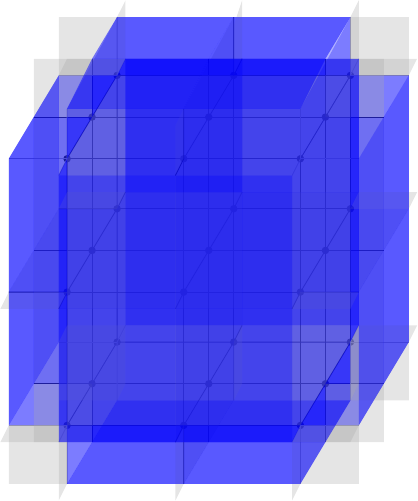}
\caption{$2\gamma+\delta$}
\label{fig:image_12}
\end{subfigure}
\caption{Critical levels of hexane from $\alpha+2\beta$ to $2\gamma+\delta$.
The new cells are plotted in blue at each stage, and all prior cells are plotted in gray. 
}
\label{fig:Filled_cubes}
\end{figure*}

At the lowest critical energy level $3\alpha$ we pass the unique global minium.
At the energy level $2\alpha+\beta$ we pass six more local minima.
At the energy level $\alpha+2\beta$ in Figure~\ref{fig:Filled_cubes}(A), we pass twelve new local minima, increasing the number of connected components to nineteen.
At the level $3\beta$ in Figure~\ref{fig:Filled_cubes}(B), the last eight connected components (vertices) appear, increasing the number of connected components to twenty-seven.
These twenty-seven connected components are arranged in the shape of a $3\times 3 \times 3$ grid on the 3-dimensional torus $(S^1)^3$.

At energy level $2\alpha + \gamma$ in Figure~\ref{fig:Filled_cubes}(C), the first six critical points of index $1$ appear, drawn as edges, reducing the number of connected components from twenty-seven down to twenty-one.

At energy level $\alpha + \beta + \gamma$ in Figure~\ref{fig:Filled_cubes}(D), twenty-four edges appear.
Together, these twenty-four edges reduce the number of connected components by twelve (from twenty-one connected components down to nine), and they produce the first twelve fundamental group generators.
This gives the first non-trivial $\pi_1$ in a sublevelset, namely 
$$
\pi_1(f_3^{-1}(-\infty, \alpha + \beta + \gamma])\cong \ast_{12} \bZ,
$$
the free group on twelve generators.

We denote a cell in our cubical complex as $xyz$, where each of $x$, $y$, and $z$ are an element of the set $\{a,b_1,b_2,c_1,c_2,d\}$.
The dimension of such a cell is the number of entries of the form $c_i$ or $d$.
For example, $aaa$ and $b_1ab_1$ are each 0-cells since they contain no copies of $c_i$ or $d$; see Figure~\ref{fig:Cubes}.
For example, $c_1ab_1$ is the 1-cell connecting the vertex $b_1ab_1$ to the vertex $aab_1$.
And $c_1ac_1$ is the 2-cell, or square, whose boundary is shown in Figure~\ref{fig:generators}(left).
With this notation, the twelve $\pi_1$ generators at energy level $\alpha+\beta+\gamma$ are grouped with four in the $xy$ plane, four in the $xz$ plane, and four in the $yz$ plane, supported on the following vertex sets:
\small
\begin{align*}
&xy: \{aaa, b_1aa, b_1b_1a, ab_1a\}, \{aaa, b_1aa, b_1b_2a, ab_2a\}, \{aaa, b_2aa, b_2b_1a, ab_1a\}, \{aaa, b_2aa, b_2b_2a, ab_2a\}\\
&xz: \{aaa, b_1aa, b_1ab_1, aab_1\}, \{aaa, b_1aa, b_1ab_2, aab_2\}, \{aaa, b_2aa, b_2ab_1, aab_1\}, \{aaa, b_2aa, b_2ab_2, aab_2\}\\
&yz: \{aaa, ab_1a, ab_1b_1, aab_1\}, \{aaa, ab_1a, ab_1b_2, aab_2\}, \{aaa, ab_2a, ab_2b_1, aab_1\}, \{aaa, ab_2a, ab_2b_2, aab_2\}.
\end{align*}
\normalsize


At the level $2\alpha + \delta$ in Figure~\ref{fig:Filled_cubes}(E), three edges appear and three new $\pi_1$ generators are born.
These are infinite cycles that never die, as they are generators for the fundamental group of the entire 3-dimensional torus. 

\begin{figure}[htb]
\includegraphics[width=0.6\textwidth]{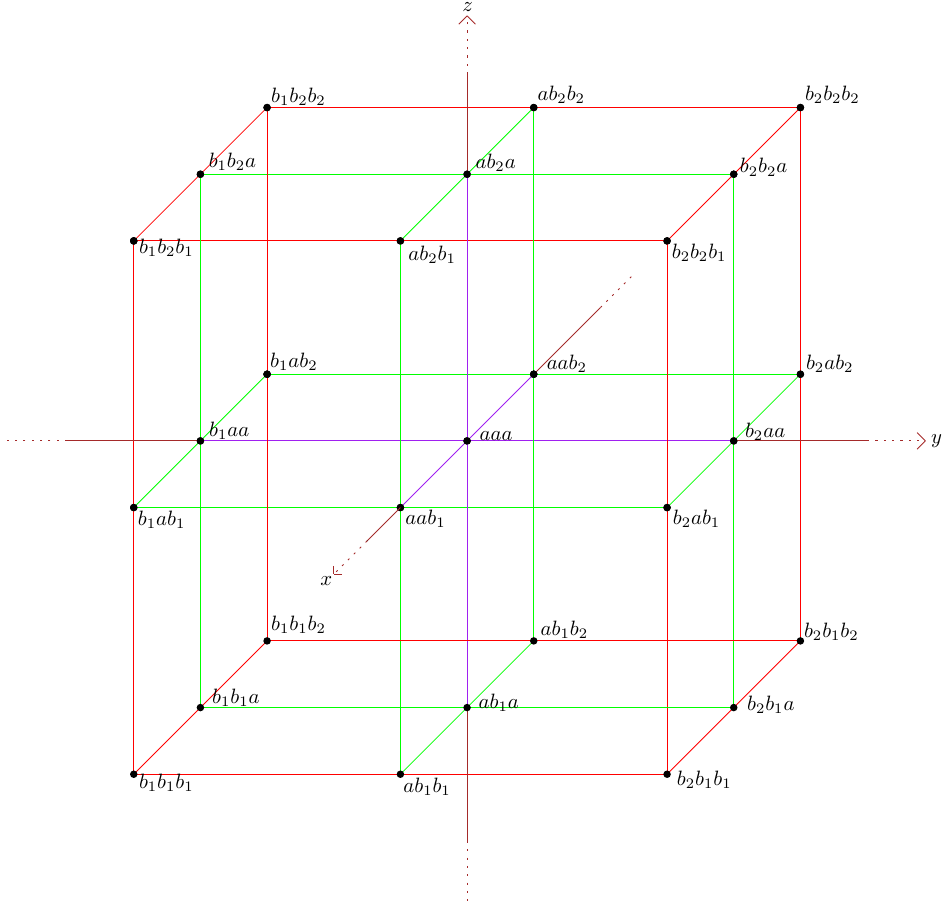}
\caption{Hexane sublevelset $f_3^{-1}((-\infty, 2 \beta + \gamma])$.
Purple edges appear at the level $2\alpha + \gamma$, green edges appear at the level $\alpha + \beta + \gamma$, brown edges appear at the level $2\alpha + \delta$ and the red edges appear at the level $2\beta + \gamma$.}
\label{fig:Cubes}
\end{figure}

At the level $2\beta +\gamma$ in Figure~\ref{fig:Filled_cubes}(F), twenty-four edges appear.
Eight of these edges connect disconnected components, and we arrive at a connected sublevelset that will remain connected for the rest of the filtration.
The remaining sixteen of these edges produce new $\pi_1$ generators.
The sublevelset at this level is homotopy equivalent to $\vee_{31}S^1$.
See Figure~\ref{fig:Cubes} for a labeled version of this sublevelset.

At the level $\alpha+\beta+\delta$ in Figure~\ref{fig:Filled_cubes}(G), twelve edges appear: four connecting the bottom to the top, four connecting the right to the left, and four connecting the front to the back.
Twelve new $\pi_1$ generators are born, which are parallel copies, lying in translates of the coordinate planes, of the $1$-cycles that were born at the level $2\alpha+\delta$.
This sublevelset is homotopy equivalent to $\vee_{43}S^1$.

At the level $2\beta+\delta$ in Figure~\ref{fig:Filled_cubes}(H), twelve edges appear connecting between vertices of the form $b_ib_jb_k$ for $i,j,k\in\{0,1\}$.
These twelve edges produce twelve new $\pi_1$ generators.
The sublevelset is homotopy equivalent to $\vee_{55}S^1$.

At the level $\alpha + 2 \gamma$ in Figure~\ref{fig:Filled_cubes}(I), the first twelve 2-cells appear.
These each kill $\pi_1$ generators born at level $\alpha + \beta + \gamma$.
For example, the $\pi_1$ generator given in Figure~\ref{fig:generators}(left) is filled in by the $2$-cell $c_1ac_1$.
This sublevelset is homotopy equivalent to $\vee_{43}S^1$.


At the level $\beta + 2 \gamma$ in Figure~\ref{fig:Filled_cubes}(J), twenty-four more 2-cells appear.
Sixteen of these 2-cells kill $\pi_1$ generators  born at level $2\beta + \gamma$.
The remaining eight 2-cells create $2$-dimensional spheres, given by the eight hollow cubes arranged in a $2 \times 2 \times 2$ grid.
This sublevelset is homotopy equivalent to $(\vee_{27}S^1)\vee(\vee_{8}S^2)$.
Hence $\pi_1 \cong \ast_{27} \bZ$.
We remark that $\pi_2((\vee_{27}S^1)\vee(\vee_{8}S^2))$ is complicated, 
and not even finitely generated.
Indeed, the universal cover of $(\vee_{27}S^1)\vee(\vee_{8}S^2)$ is homotopy equivalent to $\vee_{\infty}S^2$, the wedge sum of countably-infinite many copies of the sphere $S^2$.
By Hurewicz, $\pi_2(\vee_{\infty}S^2)$ is isomorphic $\oplus_\infty \bZ$, the direct sum of countably-infinitely many copies of $\bZ$.
And since a space and its universal cover have isomorphic homotopy groups $\pi_i$ for $i\ge 2$, this shows that $\pi_2((\vee_{27}S^1)\vee(\vee_{8}S^2)) \cong \oplus_\infty \bZ$.

At the level $\alpha + \gamma + \delta$ in Figure~\ref{fig:Filled_cubes}(K), twelve more 2-cells appear, and the twelve $\pi_1$ generators that were born at the level $\alpha + \beta + \delta$ are merged with the three infinite cycles that were born at the level $2\alpha + \delta$.
These mergings can be seen as homotopies between circles through the new blue 2-cells.
The homotopy type at this level is $(\vee_{15}S^1)\vee(\vee_{8}S^2)$, and hence $\pi_1 \cong \ast_{15} \bZ$ and $\pi_2 \cong \oplus_\infty \bZ$.

At the level $\beta + \gamma + \delta$ in Figure~\ref{fig:Filled_cubes}(L), twenty-four 2-cells appear.
As a result, twelve $\pi_1$ generators that were born at the level of the $2\beta + \delta$ are merged with the $3$ infinite cycles that were born at the level $2\alpha + \delta$.
Also at this level, twelve new $\pi_2$ generators ($2$-cycles) are born, again as hollow cubes.
Four of these hollow cubes arise from the identification of the top and bottom faces, four from the identification of the right and left faces and four from the identification of the front and back faces.
The homotopy type at this level is $(\vee_{3}S^1)\vee(\vee_{20}S^2)$, and hence $\pi_1 \cong \ast_3 \bZ$ and $\pi_2 \cong \oplus_\infty \bZ$.

At the level $\alpha + 2 \delta$ in Figure~\ref{fig:Filled_cubes}(M), we have three new 2-cells, which give three new commutator relations, $[a, b]=aba^{-1}b^{-1}$, $[a, c]=aca^{-1}c^{-1}$ and $[b, c]=bcb^{-1}c^{-1}$.
Here, $a$, $b$, and $c$ are the three $\pi_1$ generators that never die (born at level $2\alpha+\delta$).
Hence, at the level $\alpha + 2 \delta$ the fundamental group becomes abelian, namely $\pi_1 \cong  \oplus_{3}\bZ$.
The homotopy type at this level is $((S^1)^3\setminus\{p\})\vee(\vee_{20}S^2)$, the wedge sum of a $3$-dimensional torus with a point removed along with twenty $2$-spheres.

At the level $3\gamma$ in Figure~\ref{fig:Filled_cubes}(N), the first eight $3$-cells appear.
These $3$-cells are of the form $c_ic_jc_k$ for $i,j,k\in\{1,2\}$, and that is why these $3$-cells are arranged in a $2\times 2\times 2$ grid.
Each $3$-cell kills one of the $\pi_2$ generators born at level $\beta + 2 \gamma$.
The homotopy type at this level is $((S^1)^3\setminus\{p\})\vee(\vee_{12}S^2)$.

At the level $\beta + 2 \delta$ in Figure~\ref{fig:Filled_cubes}(O), the last six $2$-cells appear, which are of the form $b_idd$ and all permutations thereof.
These give birth to six new ``essential'' $\pi_2$ generators.
One of these generators can be seen in Figure~\ref{fig:generators}(right).
The homotopy type at this level is $((S^1)^3\setminus\{p\})\vee(\vee_{18}S^2)$.
At this level, the $2$-skeleton of the final simplicial complex is now complete, and therefore $\pi_1\cong \oplus_3 \bZ$ remains unchanged for the remainder of the filtration.

\begin{figure}[htb]
\includegraphics[width=0.3\textwidth]{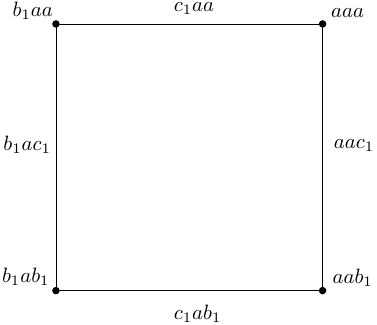}
\hspace{10mm}
\includegraphics[width=0.35\textwidth]{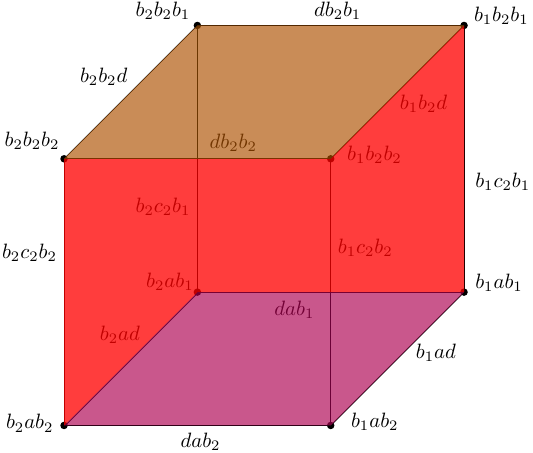}
\caption{
(Left) A $\pi_1$ generator at the level $\alpha+\beta+\gamma$ with connecting edges labeled.
(Right) One of the $\pi_2$ generators at the level $\beta+2 \delta$.
Red-colored faces (left, right, back and front faces) appear at level $\beta + \gamma + \delta$, the purple-colored face (bottom face) appears at level $\alpha + 2 \delta$, and the brown-colored face (top face) appears at level $\beta + 2 \delta.$}
\label{fig:generators}
\end{figure}

At the level $2\gamma+\delta$ in Figure~\ref{fig:Filled_cubes}(P), twelve more $3$-cells appear, of the form $c_ic_jd$ and all permutations thereof.
These kill the twelve $\pi_2$ generators that were born at the level $\beta+\gamma+\delta$.
The homotopy type at this level is $((S^1)^3\setminus\{p\})\vee(\vee_{6}S^2)$.

At the level $\gamma + 2 \delta$, six more $3$-cells appear:
$c_1dd$, $c_2dd$, $dc_1d$, $dc_2d$, $ddc_1$, $ddc_2$.
These six $3$-cells kill the six $\pi_2$ generators born at the level $\beta + 2 \delta$.
The homotopy type at this level is $(S^1)^3\setminus\{p\}$, namely a $3$-dimensional torus with a single point removed.
This is known in the literature as the \emph{spine} of a $3$-manifold~\cite{osborne1974group,stevens1975classification}.

Finally, at the level $3 \delta$, the final $3$-cell $ddd$ appears.
The sublevelset becomes the entire $3$-dimensional torus $(S^1)^3$.
We know that the homotopy groups of $(S^1)^3$ are $\pi_1((S^1)^3)=\oplus_3\bZ$ and $\pi_i((S^1)^3)=0$ for $i\ge 2$, since the homotopy group of a product is the product of the homotopy groups, i.e.\ since $\pi_i((S^1)^3)=\prod_3\pi_i(S^1)$.


Notice that considering homotopy groups in addition to homology groups of the sublevelset as in \cite{Mirth2021} shows us the difference between a wedge of circles and spheres and the product of circles.  
If one only considers persistent homology barcodes, then the barcodes obtained from the sublevelset persistent homology of hexane, as described in \cite{Mirth2021}, could potentially have also been obtained from the filtration of a space that is homotopy equivalent to the wedge sum of three circles, three $2$-dimensional spheres, and a single $3$-dimensional sphere.
However, by also considering persistent \emph{homotopy} groups, we learn that such a wedge sum is not correct and that, instead, hexane is a filtration of the $3$-dimensional torus (a 3-fold product of circles).

\begin{question}
We have described the persistent homotopy groups of the alkane molecules $f_n \colon (S^1)^n\to\bR$, defined via $f_n(\phi_1,\phi_2,\ldots,\phi_n)=f_1(\phi_1)+f_1(\phi_2)+\ldots+f_1(\phi_n)$, in the case of butane ($n=1$), penatane ($n=2$), and hexane ($n=3$).
We ask if it is possible to give a description of these persistent homotopy groups for all values of $n$.
It is conceivable that this goal is within reach; indeed~\cite{Mirth2021} provided an analytical description of the persistent \emph{homology} groups of $f_n \colon (S^1)^n\to\bR$ for all values of $n$.
\end{question}


\FloatBarrier
	
	\bibliographystyle{amsplain}
	\providecommand{\bysame}{\leavevmode\hbox
		to3em{\hrulefill}\thinspace}
	\providecommand{\MR}{\relax\ifhmode\unskip\space\fi MR }
	\MRhref  \MR
	\providecommand{\MRhref}[2]{
		\href{http://www.ams.org/mathscinet-getitem?mr=#1}{#2}
	} \providecommand{\href}[2]{#2}

\end{document}